\documentclass[11pt]{scrartcl}

\usepackage{amssymb, amsmath, amsfonts, amsthm, graphics, mathrsfs, enumerate}
\usepackage{hyperref}
\hypersetup{
    colorlinks=true,
    linkcolor=blue,
    filecolor=magenta,      
    urlcolor=cyan,
}
\usepackage[hmargin=1 in, vmargin = 1 in]{geometry}
\usepackage{graphicx}
\usepackage{caption}
\usepackage{color}
\renewcommand{\d}{\mathrm{d}}
\newcommand{\sump}{\sideset{}'\sum}

\newcommand{\CC}{\mathbb{C}}
\newcommand{\RR}{\mathbb{R}}
\newcommand{\ZZ}{\mathbb{Z}}
\newcommand{\cD}{\mathcal{D}}
\newcommand{\cG}{\mathcal{G}}
\newcommand{\cH}{\mathcal{H}}
\newcommand{\cS}{\mathcal{S}}
\newcommand{\cW}{\mathcal{W}}

\newcommand{\pd}{\partial}
\newcommand{\xx}{\mathbf{x}}
\newcommand{\yy}{\mathbf{y}}

\newcommand{\uu}{\mathbf{u}}

\newcommand{\nn}{\mathbf{n}}
\newcommand{\rr}{\mathbf{r}}

\newcommand{\ii}{\mathbf{i}}
\newcommand{\bs}{\boldsymbol}

\newcommand{\tQ}{\tilde Q}

\newcommand{\tH}{\tilde H}

\newtheorem{theorem}{Theorem}
\newtheorem{corollary}[theorem]{Corollary}

\newtheorem{remark}{Remark}

\title{A Unified Trapezoidal Quadrature Method for Singular and Hypersingular Boundary Integral Operators on Curved Surfaces}
\author{Bowei Wu, Per-Gunnar Martinsson}
\begin{document}

\maketitle

\paragraph{Abstract}
This paper describes a trapezoidal quadrature method for the discretization of singular and hypersingular boundary integral operators (BIOs) that arise in solving boundary value problems for elliptic partial differential equations. 
The quadrature is based on a uniform grid in parameter space coupled with the standard punctured Trapezoidal rule.
A key observation is that the error incurred by the singularity in the kernel can be expressed exactly using generalized Euler-Maclaurin formulae that involve the Riemann zeta function in 2D and the Epstein zeta functions in 3D. 
These expansions are exploited to correct the errors via local stencils at the singular point using a novel systematic moment-fitting approach.
This new method provides a unified treatment of all common BIOs (Laplace, Helmholtz, Stokes, etc.). We present numerical examples that show convergence of up to $32^\text{nd}$-order in 2D and $9^\text{th}$-order in 3D with respect to the mesh size.

\section{Introduction}
Singular integration is a fundamental problem that appears everywhere from elementary calculus to advanced scientific applications. In the integral equation setting, a boundary integral operator typically contains components that take one of the following forms
\begin{equation}
    \int_\Gamma \log|\xx-\yy|\,\varphi(\yy)\d s_\yy \qquad\qquad \int_\Gamma \frac{\varphi(\yy)}{|\xx-\yy|^\alpha}\d s_\yy
    \label{eq:singular_integrals}
\end{equation}
where $\varphi$ is a smooth function, $\Gamma$ is a curve in $\RR^2$ or a surface in $\RR^3$, and $\d s_\yy$ is the surface element at $\yy\in\Gamma$. Both integrals in \eqref{eq:singular_integrals} are singular when the target point $\xx$ lives on the surface $\Gamma$, whereas the second integral becomes \emph{hypersingular} if $\alpha$ is greater than the dimension of $\Gamma$.

Fast integral equation solvers are critical in a variety of applications, such as radar, medical imaging, multilayered media scattering, and fluid structure shape optimizations,   \cite{bonnet2020shape, gopal2022accelerated,greengard2009fast,zhang2021fast}. Surface singular integrals are the key to developing fast solvers for these applications in at least three important ways: 1) The accuracy of the overall solution can only be as good as that of the discretization of the surface integrals. With a high-order accurate quadrature method one can achieve any practical accuracy with much less effort than using a low-order method. 2) State-of-the-art fast algorithms for integral equations (e.g., Fast Multipole Methods and Fast Direct Solvers) exploit the global hierarchical rank-structure of the integral operators, so a quadrature method that preserves such rank-structure is highly preferred. A typically good choice is smooth quadratures with local corrections \cite{greengard2021_fmm_local}. 3) In order for an integral equation based approach to be robust, the quadrature used must be numerically stable so that the discrete system of equations is not substantially worse conditioned than the original integral equation.

We focus on locally-corrected trapezoidal quadratures, which are highly effective for the singular integrals \eqref{eq:singular_integrals} when $\Gamma$ can be 
accurately discretized using a uniform grid on a rectangle with periodic boundary conditions.
The past decades have seen many successful trapezoidal quadratures designed for integral equation problems in $\RR^2$, such as \cite{kapur1997high,alpert1999hybrid} for logarithmically or algebraically singular line integrals, \cite{sidi2013compact} for hypersingular integrals, and \cite{duan2009high} for 2D volume integrals with the Helmholtz kernel. In contrast, effective trapezoidal quadratures for integral equation problems in $\RR^3$ have been lacking. There are accurate double- and triple-trapezoidal quadratures for the Coulomb potential, such as \cite{aguilar2005high,marin2014corrected}, but they do not account for surface integrals.

All existing locally-corrected trapezoidal quadratures have been derived based on the fact that the trapezoidal rule, when applied to approximating singular integrals, has an asymptotic error expansion whose coefficients can be expressed using the Riemann zeta function or its derivative; this fact was first discovered as generalized Euler-Maclaurin formulae in \cite{navot1961extension,navot1962further}, see \cite{sidi2018recent} for a most general formula. 

The current work is evolution of our prior papers 
\cite{wu2021zeta,wu2021corrected}.
In \cite{wu2021zeta} we generalized the Euler-Maclaurin formula to handle line integrals by including the geometric information into the error expansion, which resulted in simple trapezoidal quadratures (called ``zeta quadrature'') for logarithmically singular kernels that are stable for arbitrarily high-order corrections. 
We generalized the zeta quadrature to surface integrals in \cite{wu2021corrected}, where we demonstrated that the double-trapezoidal rule applied to singular integrals also has an asymptotic error expansion in terms of zeta functions, called \emph{Epstein zeta functions} \cite{epstein1903theorie,epstein1906theorie}; this error expansion captures the local surface distortion by taking into account the derivatives of the geometric parameterization. This discovery allowed us to derive zeta quadratures that are $5^\text{th}$ order accurate for the Laplace and Helmholtz kernels. However, our approach in \cite{wu2021corrected} requires higher-order derivatives of the geometry for higher-order corrections, which makes a quadrature of $7^\text{th}$ order or higher impractical.

In this paper, we introduce a new approach to derive zeta quadratures for surface integrals. This new approach treats kernels with different singularities in a unified manner (including hypersingular kernels), does \emph{not} require higher derivatives of the geometry other than the first fundamental forms, and calculates the correction weights systematically using a moment-fitting approach. As a result, the derivation of higher-order quadrature formulae is significantly simplified and the method can be easily applied to different kernels. As a trade-off of using only the first derivatives of the geometry, the correction stencils for higher-order quadratures are generally bigger than the corresponding stencils in \cite{wu2021corrected}.

This paper is organized as follows. First, in Section \ref{sc:zeta2d} we introduce the zeta quadrature for line integrals, which is based on the 1D generalized Euler-Maclaurin formulae. Then in Section \ref{sc:zeta3d} we introduce a 2D generalized Euler-Maclaurin formula and develop the new zeta quadrature that can handle a variety of surface integral operators with a unified approach. Numerical results are presented in both sections. Finally, we conclude in Section \ref{sc:conclusion}.

%ssssssssssssssssssssssssss
\section{Zeta quadrature for line integrals in $\RR^2$}
\label{sc:zeta2d}

In this section, we introduce the zeta quadrature for line integrals in $\RR^2$, which serves as a motivation for the unified quadrature method for surface integrals in $\RR^3$ in Section \ref{sc:zeta3d}. We first review Sidi's results on the generalized Euler-Maclaurin formula \cite{sidi2012algebraic,sidi2018recent} and hypersingular trapezoidal quadrature \cite{sidi2013compact}, and derive the quadrature formulae for the Laplace and Helmholtz hypersingular integral operators on a contour. In particular we show that in certain circumstances, an algebraically high-order quadrature rule can attain any practically relevant accuracy using fewer degrees of freedom than a spectrally accurate rule (e.g., see Fig.\ref{fig:hyper2d_lap_conv}). 

We start by stating the following generalized Euler-Maclaurin formula from \cite[Theorem 2.3]{sidi2012algebraic}, which is the theoretical basis for the treatment of practically all boundary integral operators for elliptic PDEs.

\begin{theorem}\label{thm:EM1d}
Let $f\in C^\infty(a,b)$ and assume that $f(x)$ has the asymptotic expansions
\begin{equation}
\begin{aligned}
f(x) &\sim K(x-a)^{-1} + \sum_{k=0}^\infty c_k(x-a)^{\alpha_k} & &\text{as }x\to a^+\\
f(x) &\sim L(b-x)^{-1} + \sum_{k=0}^\infty d_k(b-x)^{\beta_k} & &\text{as }x\to b^-
\end{aligned}
\end{equation}
where $K, L$, and $c_k, d_k$ are constants, and where $\alpha_k$ and $\beta_k$ are distinct complex numbers that are different from $-1$ and satisfy
\begin{equation}
\begin{aligned}
\mathrm{Re}\,\alpha_0\leq\mathrm{Re}\,\alpha_1\leq\mathrm{Re}\,\alpha_2\leq\dots;\quad \lim_{k\to\infty}\mathrm{Re}\,\alpha_k=\infty,\\
\mathrm{Re}\,\beta_0\leq\mathrm{Re}\,\beta_1\leq\mathrm{Re}\,\beta_2\leq\dots;\quad \lim_{k\to\infty}\mathrm{Re}\,\beta_k=\infty.
\end{aligned}
\end{equation}
Let $h = (b-a)/n$ for $n \in\mathbb{Z}^+$. Then, as $n\to\infty$ (or $h\to0$),
\begin{equation}
\begin{aligned}
h\sum_{j=1}^{n-1}f(a+jh) \sim \int_a^b f(x)\,\d x & + K(\gamma-\log h) + \sum_{k=0}^\infty c_k\zeta(-\alpha_k)h^{\alpha_k+1}\\
& + L(\gamma-\log h) + \sum_{k=0}^\infty d_k\zeta(-\beta_k)h^{\beta_k+1},
\end{aligned}
\label{eq:extendedEM}
\end{equation}
where $\gamma$ is Euler-Mascheroni constant, $\zeta(s)$ is the Riemann zeta function, and the integral exists in the \emph{Hadamard finite-part} sense.
\end{theorem}

%rrrrrrrrrrrrrrrr
\begin{remark}
Notice that when $f(x)$ is a regular function, then $K=L=0$, $\alpha_k = \beta_k= k$, $c_k = \tfrac{f^{(k)}(a)}{k!}$, and $d_k = \tfrac{(-1)^{k}f^{(k)}(b)}{k!}$, so Theorem 1 reduces to the classical Euler-Maclaurin formula
\begin{equation}
h\sum_{j=1}^{n-1}f(a+jh) \sim \int_a^b f(x)\,\d x -\frac{h}{2}[f(a)+f(b)] +  \sum_{k=1}^\infty \frac{B_{2k}}{(2k)!}[f^{(2k-1)}(b)-f^{(2k-1)}(a)]h^{2k},
\end{equation}
where $B_k$ are the Bernoulli numbers. Here we have used the fact that $\zeta(-k) = (-1)^k \frac{B_{k+1}}{k+1}$ for $k\geq0$, under the convention $B_1 = -1/2$.
\end{remark}

\begin{remark}
\cite{sidi2012algebraic_log} proposed a more general E-M formula that unifies the treatment of algebraic and logarithmic singularities. We focus on algebraic singularities in this paper for conciseness; for the treatment of logarithmically singular kernels, see \cite{wu2021zeta}.
\end{remark}

Next, we consider a function $f$ that is periodic on the interval $[-a,a)$, and infinitely differentiable everywhere except at $0$. We further assume that 
\begin{equation}
f(x) \sim \frac{\varphi(x)}{x^2}\quad \text{as} \quad x\to0,
\end{equation}
where $\varphi(x)$ is analytic near $0$.
Denote the finite-part integral
\begin{equation}
I[f] := \int_{-a}^af(x)\,\d x
\end{equation}
Let $h = a/n$ and denote the punctured periodic trapezoidal rule
\begin{equation}
T_n[f] := h\sum_{\substack{i=-n\\i\neq0}}^{n-1} f(ih).
\end{equation}

\begin{corollary}[Sidi 2013]
Suppose $f(x)$ is a periodic hypersingular function as defined above, then
\begin{equation}
I[f] = T_n[f]  - \frac{\pi^2}{3h}\varphi(0) + \frac{\varphi''(0)h}{2} + O(h^\mu)
\label{eq:zetatrap_x2}
\end{equation}
for any $\mu>0$.
\end{corollary}
\begin{proof}
By applying Theorem \ref{thm:EM1d} to $f(x)$ separately on $(-a,0)$ and $(0,a)$, then combining them, we obtain
\begin{equation}
\begin{aligned}
T_n[f] &\sim  I[f]  + (\varphi'(0)-\varphi'(0))\cdot(\gamma -\log h)\\
&\quad + \sum_{\substack{k=0\\k\neq1}}^\infty\Big(\frac{\varphi^{(k)}(0)}{k!} + (-1)^k\frac{\varphi^{(k)}(0)}{k!}\Big) \zeta(2-k)h^{k-1}\\
&\quad -\frac{h}{2}[f(-a)+f(a)] +  \sum_{k=1}^\infty \frac{B_{2k}}{(2k)!}[f^{(2k-1)}(a)-f^{(2k-1)}(-a)]h^{2k}\\
&\sim I[f]  + \sum_{k=0}^\infty\frac{\varphi^{(2k)}(0)}{(2k)!}\,2\zeta(2-2k)h^{2k-1}
\end{aligned}
\label{eq:thm1_deriv}
\end{equation}
where all the boundary terms involving $f^{(k)}(\pm a)$ vanish because of periodicity, then the desired trapezoidal rule \eqref{eq:zetatrap_x2} is obtained by substituting $\zeta(2) = \pi^2/6$, $\zeta(0) = -1/2$ and $\zeta(-2k)=0$ for all $k>0$. Clearly this approximation converges faster than any algebraic order of $h$. 
\end{proof}

\begin{remark}
In general, when $f(x)$ is not periodic, one can keep the boundary terms in \eqref{eq:thm1_deriv} to obtain a high-order correction.
\end{remark}

When the formula for the second derivative $\varphi''(0)$ in \eqref{eq:zetatrap_x2} is not available, which is generally the case in practice, we use a high-order central difference approximation
$$
\varphi''(0) = \frac{1}{h^2}\sum_{j=-M}^Mc_j\varphi(jh)+ O(h^{2M}),
$$
where the coefficients $c_j, j = -M,\dots,M$, can be calculated by solving the system 
\begin{equation}
c_m \equiv c_{-m}\quad \text{and}\quad  \sum_{j=-M}^Mc_j j^{2m} = 2\delta_{1m},\quad \text{for } m = 0,\dots,M.
\end{equation} This yields the following high-order quadrature
\begin{equation}
I[f] = T_n[f]  - \frac{\pi^2}{3h}\varphi(0) + \frac{1}{2h}\sum_{j=-M}^Mc_j\varphi(jh)+ O(h^{2M+1}).
\label{eq:zetatrap_x2_cd}
\end{equation}
Alternatively, one can eliminate the $\varphi''(0)$ term in \eqref{eq:zetatrap_x2} via Richardson extrapolation \cite{sidi2013compact}, resulting in a quadrature on an alternating grid
\begin{equation}
I[f] = \tilde{T}_n[f]  - \frac{\pi^2}{2h}\varphi(0) + O(h^\mu),\qquad \forall \mu>0.
\label{eq:zetatrap_x2_alt}
\end{equation}
where
$$
\tilde{T}_n[f] = 2h\sum_{\substack{i=-n\\i \text{ odd}}}^{n-1} f(ih).
$$
We next present results of applying the quadrature \eqref{eq:zetatrap_x2_cd} to integral operators of Laplace (Section \ref{sc:lap2d}), Helmholtz (Section \ref{sc:helm2d})  and Stokes (Appendix \ref{app:sto2d}).

%ssssssssssssssssssssssssss
\subsection{Zeta quadrature for the Laplace hypersingular operator} 
\label{sc:lap2d}
Consider the Laplace hypersingular potential on a smooth simple closed curve $\Gamma\subset\RR^2$
\begin{equation}
\cH[\sigma](\xx):=\int_\Gamma K(\xx,\yy)\sigma(\yy)\,\d s_\yy\equiv\frac{1}{2\pi}\int_\Gamma \Big(\frac{\nn_\xx\cdot \nn_\yy}{r^2}-2\mu_\xx\mu_\yy\Big)\,\sigma(\yy)\,\d s_\yy,\qquad \xx\in\Gamma
\label{eq:hyper2d_lap}
\end{equation}
where $r:=|\xx-\yy|$, $\nn_\xx$ denotes the unit outward normal at $\xx$; the terms $\mu_\xx := {(\xx-\yy)\cdot \nn_\xx}/{r^2}$ and $\mu_\yy := {(\xx-\yy)\cdot \nn_\yy}/{r^2}$ are smooth as $\yy\to\xx$. Without loss of generality, assume that $\Gamma$ is parameterized by a smooth periodic function $\bs\rho:[-a,a)\mapsto\Gamma$ such that $\xx = \bs\rho(0)$, then the hypersingular component of \eqref{eq:hyper2d_lap} is given by
\begin{equation}
\int_\Gamma\frac{\nn_\xx\cdot \nn_\yy}{2\pi r^2}\sigma(\yy)\,\d s_\yy \equiv  \int_{-a}^a\frac{(\nn(0)\cdot \nn(x))\,|\bs\rho'(x)|}{2\pi\,r(x)^2}\sigma(x)\,\d x,
\label{eq:lap_hyper_comp}
\end{equation}
where, using the parameterization $\yy=\bs\rho(x)$, it is understood that $\sigma(x)\equiv\sigma(\bs\rho(x))=\sigma(\yy)$ and similarly for other terms. Let $\bs\rho'_0:=\bs\rho'(0)$ and define a ``local bending'' term
\begin{equation}
B(x):=\frac{r(x)^2-|\bs\rho'_0 x|^2}{|\bs\rho'_0 x|^2},
\label{eq:local_bending_1d}
\end{equation}
which can be understood as the relative deviation of the extrinsic measure $r(x)^2\equiv|\bs\rho(x)-\bs\rho(0)|^2$ from the intrinsic measure $|\bs\rho'_0 x|^2$ near $0$. Note that $B(x)$ is smooth and $B(x)=O(x)$ as $x\to0$, so
\begin{equation}
\frac{1}{r(x)^2} = \frac{1}{|\bs\rho'_0x|^2\,(1+B(x))} = \frac{1}{|\bs\rho'_0x|^2}\sum_{k=0}^\infty (-1)^kB(x)^k = \frac{1-B(x)+B(x)^2}{|\bs\rho'_0x|^2}+O(x).
\label{eq:geo_expan}
\end{equation}
Then the integrand of \eqref{eq:lap_hyper_comp} becomes
$$
\frac{(\nn(0)\cdot \nn(x))|\bs\rho'(x)|}{2\pi\,r(x)^2} \sim \frac{g(x)}{x^2}+O(x),\quad  x\to0
$$
where
\begin{equation}
g(x):=\frac{(\nn(0)\cdot \nn(x))\,|\bs\rho'(x)|}{2\pi\,|\bs\rho'_0|^2}(1-B(x)+B(x)^2).
\end{equation}
Applying the quadrature \eqref{eq:zetatrap_x2_cd} to \eqref{eq:lap_hyper_comp} with $h=2a/N$ and $\varphi(x) = g(x)\sigma(x)$ yields
\begin{equation}
\begin{aligned}
\cH[\sigma](\xx) = \sum_{\substack{j=-N/2\\j\neq0}}^{N/2}K(\bs\rho(0),\bs\rho(jh))\sigma(jh)w_j + \Big(\frac{\lambda_0^2w_0}{4\pi}- \frac{\pi}{6w_0}\Big)\sigma(0) \\ +\frac{1}{2h}\sum_{j=-M}^Mc_jg(jh)\sigma(jh) +O(h^{2M+1})
\end{aligned}
\label{eq:zetatrap_lap_cd}
\end{equation}
where $w_j:=|\bs\rho'(jh)|\,h$ and where $\lambda_0$ is the curvature of $\Gamma$ at $\bs\rho(0)$ such that the $\lambda_0^2$ term is the diagonal limit of the $\mu_\xx\mu_\yy$ component of $K(\xx,\yy)$.

On the other hand, applying the quadrature \eqref{eq:zetatrap_x2_alt} yields
\begin{equation}
\cH[\sigma](\xx) = \sum_{\substack{j=-N/2\\j\text{ odd}}}^{N/2}K(\bs\rho(0),\bs\rho(jh))\sigma(jh)w_j - \frac{\pi\sigma(0)}{4w_0} +O(h^p),\quad \forall p>0
\label{eq:zetatrap_lap_alt}
\end{equation}

To compare the convergence of the above two formulae, we solve the Dirichlet problems for the Laplace equation
\begin{equation}
\begin{cases}
\Delta u = 0\;\text{in }\Omega,\quad u = \sigma\;\text{ on }\Gamma & \text{(interior)}\\
\Delta u = 0\;\text{in }\Omega^c,\quad u = \sigma\;\text{ on }\Gamma,\quad u\to\frac{\Sigma}{2\pi}\log|\xx| + \omega\;\text{ as }|\xx|\to\infty & \text{(exterior)}
\end{cases}
\end{equation}
where $\Sigma$ and $\omega$ are known constants. Using with the ansatz
$$
u(\xx) =  \begin{cases}
\cS[\tau](\xx) - \cD[\sigma](\xx) & \xx\in\Omega\text{ (interior)}\\
\cD[\sigma](\xx) - \cS[\tau](\xx) + \omega & \xx\in\Omega^c\text{ (exterior)},
\end{cases}
$$
the unknown Neumann data $\tau = \frac{\partial u}{\partial \nn}$ on $\Gamma$ is the solution of the BIE\cite{hsiao2008boundary}
\begin{equation}
\begin{cases}
(-\tfrac{1}{2}+\cD^*)\tau = \cH\sigma &\text{ (interior)}\\
(\tfrac{1}{2}+\cD^*)\tau + \int_\Gamma \tau\,\d s = \cH\sigma + \Sigma &\text{ (exterior)}
\end{cases}
\label{eq:bie_lap2d}
\end{equation}
Figure \ref{fig:hyper2d_lap_conv} shows the convergence results solving the BIEs, where $\cH\sigma$ is discretized using a $32^\text{nd}$-order central-difference quadrature \eqref{eq:zetatrap_lap_cd} as well as the alternating-grid spectral quadrature \eqref{eq:zetatrap_lap_alt}. Even though the alternating-grid quadrature is spectrally convergent, it uses only half of the available information for each target point and has bigger absolute errors than the central-difference zeta quadrature.

%ffffffffffffffffff
\begin{figure}[htbp]
\centering
\includegraphics[width=\textwidth]{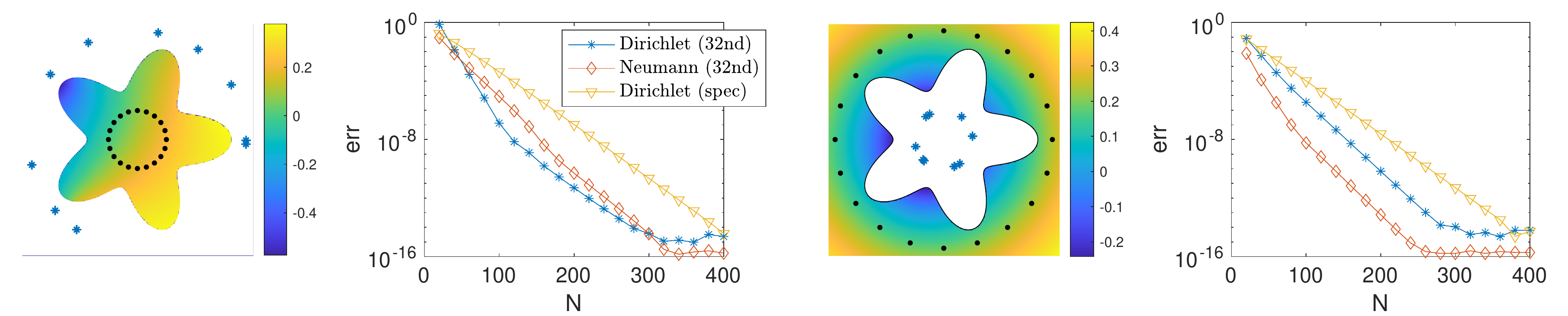}
\caption{Use the zeta quadrature rules to solve all four Laplace BVPs (Dirichlet/Neumann, interior/exterior) via the Green's representation integral equation formulation \eqref{eq:bie_lap2d}. Exact solutions are generated by point sources whose locations are shown by the blue asterisks, the $\infty$-norm errors of the solutions are measured at test points shown by the block dots. For the Dirichlet problems, results using a $32^\text{nd}$-order quadrature \eqref{eq:zetatrap_lap_cd} and using the spectral zeta quadrature \eqref{eq:zetatrap_lap_alt} are compared. It is observed that the algebraically high-order zeta quadrature uses a smaller number of degrees of freedom than the spectral quadrature to reach any given practical accuracy.}
\label{fig:hyper2d_lap_conv}
\end{figure}

%ssssssssssssssssssssssssss
\subsection{Zeta quadrature for the Helmholtz hypersingular operator}
\label{sc:helm2d}
Let $\kappa$ be the wave number, the Helmholtz hypersingular integral operator is
\begin{equation}
\begin{aligned}
\cH_\kappa[\sigma](x) &:= \int_\Gamma K_\kappa(\xx,\yy)\sigma(\yy)\,\d s_\yy\\
&\equiv \int_\Gamma \frac{i\kappa^2}{4}\big[H_0(\kappa r)\mu_\xx\mu_\yy+\frac{H_1(\kappa r)}{\kappa r}(\nn_\xx\cdot \nn_\yy-2\mu_\xx\mu_\yy)\Big]\sigma(\yy)\,\d s_\yy
\end{aligned}
\end{equation}
then with a similar derivation as for the Laplace hypersingular kernel, we have
\begin{equation}
\begin{aligned}
\cH_\kappa[\sigma](\xx)&=\sum_{\substack{j=-N/2\\j\neq0}}^{N/2}K_\kappa(\bs\rho(0),\bs\rho(jh))\sigma(jh)w_j + C_0 - \frac{\pi}{6 w_0}\sigma(0)\\
&\quad + \sum_{j=-M}^M\Big(\tilde{K}_\kappa(\bs\rho(0),\bs\rho(jh))\tilde{w}_j + \frac{1}{2h}c_jg(jh)\Big)\sigma(jh) + O(h^{2M+1})
\end{aligned}
\label{eq:zetatrap_helm_cd}
\end{equation}
where
$$
\begin{aligned}
C_0 &= \Big(\frac{\lambda_0^2}{4\pi} + \frac{i\kappa^2}{8} - \frac{\kappa^2}{4\pi}\log\frac{w_0\kappa}{2} + \frac{(1-2\gamma)\kappa^2}{8\pi}\Big)w_0\\
\tilde{K}_\kappa(x,y) &= \frac{\kappa^2}{2\pi}\Big(J_0(\kappa r)\mu_\xx\mu_\yy+\frac{J_1(\kappa r)}{\kappa r}(\nn_\xx\cdot \nn_\yy-2\mu_\xx\mu_\yy)\Big)\\
\end{aligned}
$$
and where $\tilde{w}_j$ are the zeta quadrature weights associated with the $\log\frac{1}{r}$ kernel\cite{wu2021zeta}, where the $w_0\log w_0$ term is included in $C_0$; the rest of the terms in \eqref{eq:zetatrap_helm_cd} are the same as those in \eqref{eq:zetatrap_lap_cd}.

We solve the Helmholtz equation $\Delta u + \kappa^2 u = 0$ in $\Omega^c$ with the Dirichlet condition $u=f$ or Neumann condition $u=g$ on $\Gamma$,
\begin{equation}
\begin{cases}
(\tfrac12+\cD_\kappa-i|\kappa|\cS_\kappa)\tau=f & \text{(Dirichlet)}\\
(-\tfrac12+\cD_\kappa^*+i|\kappa|\cH_\kappa)\sigma=g & \text{(Neumann)}
\end{cases}
\end{equation}
where the Neumann problem involves the hypersingular operator. Convergence results are shown in Fig. \ref{fig:hyper2d_helm_conv}.
%ffffffffffffffffff
\begin{figure}[htbp]
\centering
\includegraphics[width=\textwidth]{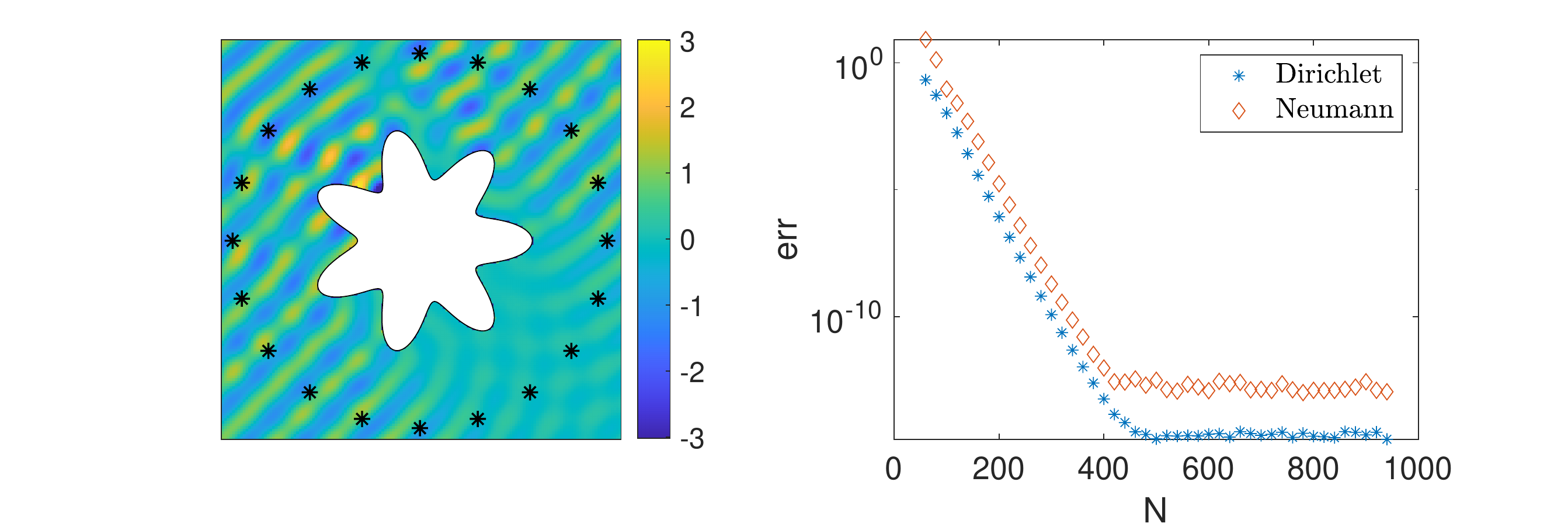}
\caption{Convergence results of solving Helmholtz exterior Dirichlet and Neumann BVPs using $32^\text{nd}$-order zeta quadratures. (a) Real part of sound-hard scattering of a planewave on a smooth obstacle. The asterisks ``$*$'' denote test locations where the error of the solution is measured.   (b) Convergence of max error at the test locations for the planewave sound-soft (Dirichlet) and sound-hard (Neumann) scattering. Note that the Neumann problem involves the hypersingular operator, resulting in absolute errors that are bigger than the Dirichlet problem.}
\label{fig:hyper2d_helm_conv}
\end{figure}

%ssssssssssssssssssssssssss
\section{Zeta quadrature for surface integrals in $\RR^3$, a unified approach}
\label{sc:zeta3d}
In this section, we describe a trapezoidal quadrature method that systematically treat integral operators on doubly-periodic surfaces, which is a generalization of the one-dimensional quadrature in Section \ref{sc:zeta2d}. 

We first set up the notations for this section. Let $\uu:=(u,v)\in\RR^2$ and $\ii:=(i,j)\in\ZZ^2$. Denote the 1-norms $|\ii|:=\max\{|i|,|j|\}$ and $|\uu|:=\max\{|u|,|v|\}$.  We will use $O(\uu^p)$ to denote terms of degree $p$ or higher in $\uu$, and $\Theta(\uu^p)$ to denote terms of degree exactly $p$.

Let $R=(a,b)\times(c,d)$ be a rectangular domain with $a<0<b$ and $c<0<d$. Let $\Gamma\subset\mathbb{R}^3$ be a smooth surface parameterized by $\bs\rho(\uu)$, $\uu\in R$. Denote the Jacobian $J(\uu):= |\bs\rho_u(\uu)\times\bs\rho_v(\uu)|$. 

For simplicity and without loss of generality, we assume that the singularities of the integrals in this section are located at $\xx=\bs\rho(\mathbf{0})$. Define the first fundamental form at $\xx$ as
\begin{equation}
Q(\uu)=Eu^2+2Fuv+Gv^2,
\end{equation}
where $E=\bs\rho_u(\mathbf{0})\cdot\bs\rho_u(\mathbf{0}),\,F=\bs\rho_u(\mathbf{0})\cdot\bs\rho_v(\mathbf{0}),\,G=\bs\rho_v(\mathbf{0})\cdot\bs\rho_v(\mathbf{0})$. 

In addition, we will need two important concepts associated with the quadratic form $Q(\uu)$. Define the \emph{Epstein zeta function} $Z(s;Q)$ as
\begin{equation}
\label{eq:epstein_zeta}
Z(s;Q) \equiv Z(s;E,F,G):=\sum_{\ii\in\ZZ^2\backslash\{\mathbf{0}\}}\frac{1}{(Ei^2+2Fij+Gj^2)^{s/2}},\quad \mathrm{Re}\,s>2
\end{equation}
which can then be extended to all $s\in\mathbb{C}\backslash\{2\}$ by analytical continuation. The Epstein zeta functions are generalizations of the Riemann zeta function to higher dimensions\cite{epstein1903theorie,epstein1906theorie}, such that the 1D analog of $Z(s;Q)$ would be $2\zeta(s)$.

Finally, define the \emph{Wigner limits} associated with $Q(\uu)^{-s/2}$ to be 
\begin{equation}
\cW^s[f]:=\lim_{h\to0}\lim_{N\to\infty}\left\{\sum_{0<|\ii|\leq N}\frac{f(\ii)}{Q(\ii)^\frac{s}{2}}\eta(\ii h) - \int_{|\uu|\leq N+\frac12}\frac{f(\uu)}{Q(\uu)^\frac{s}{2}}\eta(\uu h)\,\d\uu\right\}, \label{eq:wigner_limits}
\end{equation}
where $\eta(\uu)$ is a $C^\infty$ smooth function with compact support, which also satisfies $\eta(\mathbf{0})=1$ and $\eta(\uu)\equiv\eta(-\uu)$.

\begin{remark}
The Wigner limits were introduced by Eugene Wigner\cite{wigner1934} to calculate the electron sum of a body-centered cubic crystal system with a compensating positive charges in the background. The original definition can be written as
\begin{equation}
\lim_{N\to\infty}\left\{\sum_{0<|\ii|\leq N}\frac{1}{Q(\ii)^\frac{s}{2}} - \int_{|\uu|\leq N+\frac12}\frac{1}{Q(\uu)^\frac{s}{2}}\,\d\uu\right\}
\label{eq:wigner_limits_original}
\end{equation}
for $0<\mathrm{Re}\,s<2$ in 2D (see \cite{borwein2014lattice} for more details). In \cite{wu2021corrected}, we have generalized this definition to \eqref{eq:wigner_limits} for all $\mathrm{Re}\,s<2$, where the bump function $\eta(\uu)$ is introduced to ensure proper convergence when $\mathrm{Re}\,s\leq0$. We named \eqref{eq:wigner_limits} the ``converged Wigner-type limits'' to distinguish them from the original definition \eqref{eq:wigner_limits_original}. However, in this paper, we will simply call \eqref{eq:wigner_limits} the Wigner limits for simplicity, which should not cause any confusions; further more, we will show that this definition holds for all $s\neq2$ in the finite-part sense.
\end{remark}

%ssssssssssssssssssssssssss
\subsection{Generalized Euler-Maclaurin formula for double integrals} Common boundary integral operators for elliptic PDEs involve integrals of the form
\begin{equation}
\int_R\frac{\varphi(\xx,\yy)}{r^p}\,\d\uu,
\label{eq:generic_kernel0}
\end{equation}
where $r := |\xx-\yy| \equiv |\bs\rho(\mathbf{0})-\bs\rho(\uu)|,$ and where $\varphi$ is smooth on $\Gamma\times\Gamma$ such that
\begin{equation}
\varphi(\xx,\yy) \equiv \varphi(\bs\rho(\mathbf{0}),\bs\rho(\uu))\sim O(\uu^{2q})
\label{eq:offset}
\end{equation}
for some integer $q\geq0$. Similar to \eqref{eq:local_bending_1d}, one can define a ``local bending'' term
\begin{equation}
B(\uu):=\frac{r(\uu)^2-Q(\uu)}{Q(\uu)}
\end{equation}
which satisfies $B(\uu)\to0$ as $\uu\to\mathbf{0}$. (However, unlike \eqref{eq:local_bending_1d}, $B(\uu)$ is not smooth at $\mathbf{0}$.) Then expanding $r^{-p}$ with respect to $B(\uu)$ in \eqref{eq:generic_kernel0} gives
\begin{equation}
\frac{\varphi(\xx,\yy)}{r^p} = \varphi(\xx,\yy)\sum_{m=0}^\infty\binom{-\tfrac{p}{2}}{m}\frac{(r^2-Q)^m}{Q^{m+p/2}} = \sum_{m=0}^\infty\frac{O(\uu^{2q+3m})}{Q^{m+p/2}}.
\label{eq:generic_kernel}
\end{equation}
where the last equality uses the condition \eqref{eq:offset} and the fact that $r^2-Q = O(\uu^{3})$.

In order to integrate \eqref{eq:generic_kernel} to high-order, we need the following generalized Euler-Maclaurin formula, which is a 2D generalization of \eqref{eq:extendedEM}.

%thmmmmmmmmmmmm
\begin{theorem}
    \label{thm:extendedEM2D}
    Let $f\in C^\infty(R\setminus\{\mathbf{0}\})$, where $R=(a,b)\times(c,d)$ is a rectangular domain with $a<0<b$ and $c<0<d$. Assume that $f(x)$ is doubly-periodic or compactly supported in $R$ (i.e. $f(u,c)\equiv f(u,d)$ and $f(a,v)\equiv f(b,v)$ for all $a\leq u\leq b$ and $c\leq v\leq d$) and that $f(\uu)$ is smooth except at $\mathbf{0}$, where $f(\uu)$ has an asymptotic expansion
    \begin{equation}
    f(\uu) \sim \sum_{k=1}^\infty\sum_{l=0}^{n_k} \alpha_{n_k,l}^{s_k}\frac{u^{n_k-l}v^l}{Q(\uu)^{s_k/2}}, \quad \text{as }\uu\to\mathbf{0}
    \label{eq:sing_asymp_2D}
    \end{equation}
    where $0\leq n_k\in \ZZ$, $\alpha_{n_k,l}^{s_k}$ are coefficients for the expansion, and where $s_k\in\CC$ satisfies
    \begin{equation}
    s_k\neq2;\quad \mathrm{Re}\,s_0\leq\mathrm{Re}\,s_1\leq\mathrm{Re}\,s_2\leq\dots;\quad \lim_{k\to\infty}\mathrm{Re}\,s_k=\infty.
    \end{equation}
    Assume further that $h = (b-a)/M = (d-c)/N$ for some $M,N \in\mathbb{Z}^+$, let $u_i=a+ih,0\leq i\leq M,$ and  $v_j=c+jh, 0\leq j\leq N$, such that $u_{i_0}=v_{j_0}=0$ for some $0<i_0<M, 0<j_0<N$. Then as $h\to0$ (thus $M,N\to\infty$),
    \begin{equation}
    \sum_{\substack{0< j< M, 0< k< N\\ (i,j)\neq (i_0, j_0)}} f(u_i,v_j)\,h^2 \sim \int_R f(\uu)\,\d\uu + \sum_{\substack{k=1\\n_k=\text{even}}}^\infty\sum_{l=0}^{n_k} \alpha_{n_k,l}^{s_k}\cW^{s_k}\left[u^{n_k-l}v^l\right]h^{n_k-s_k+2}
    \label{eq:extendedEM2D}
    \end{equation}
\end{theorem}
\begin{proof}
The proof of \eqref{eq:extendedEM2D} is a simple superposition of \eqref{eq:wigner_limits} with $s = s_k$ and $f$ replaced by the monomials $u^{n_k-l}v^l, k = 1, 2, \dots.$ When $n_k$ is odd, the Wigner limits are zero because the integrand and summand are odd functions, therefore only even $n_k$'s are present in \eqref{eq:extendedEM2D}. The periodicity of $f(\uu)$ ensures the vanishing of trapezoidal rule errors on the edges of $R$.
\end{proof}

The generalized Euler-Maclaurin formula \eqref{eq:extendedEM2D} provides an error expansion for the double-trapezoidal rule approximation of integrals with $Q(\uu)^{-\frac{s}{2}}$ type singularities ($s\neq2$). Wigner limits of the form $W^{s}\left[u^{2n-l}v^l\right]$ are involved in \eqref{eq:extendedEM2D}, which when evaluated using the definition \eqref{eq:wigner_limits} lead to numerical subtraction errors as $h\to0$. Fortunately, the next theorem allows us to evaluate the Wigner limits as the values or parametric derivatives of the Epstein zeta functions, which is both computationally cheaper and more accurate. For details on the algorithms evaluating the Epstein zeta functions, we refer to \cite{wu2021corrected}.

%thmmmmmmmmmmmm
\begin{theorem}
    \label{thm:wigner}
    Suppose $f(\uu)=u^{2n-l}v^l, 0\leq l\leq 2n$, is a monomial of degree $2n$ for some integer $n\geq0$, then the Wigner limit $\cW^s[f]$ has an analytic expression given by the associated Epstein zeta function and its parametric derivatives, as follows
    \begin{equation}
    \cW^s\left[u^{2n-l}v^l\right]= 
     \begin{cases}
        \displaystyle \tfrac{\Gamma(-s/2+1)}{\Gamma(n-s/2+1)}\left(\frac{\pd}{\pd E}\right)^{n-l} \left(\frac{1}{2}\frac{\pd}{\pd F}\right)^lZ\big(s-2n;Q\big), & l\leq n;\\[15pt]
        \displaystyle \tfrac{\Gamma(-s/2+1)}{\Gamma(n-s/2+1)}\left(\frac{1}{2}\frac{\pd}{\pd F}\right)^{2n-l}\left(\frac{\pd}{\pd G}\right)^{l-n}Z\big(s-2n;Q\big), & l\geq n.
      \end{cases}
    \label{eq:wigner_epstein}
    \end{equation}
    On the other hand, when $f(\uu)$ is a monomial of odd degrees, then $\cW^s[f]\equiv0$.
\end{theorem}

\begin{proof}
When $f(\uu)$ is a monomial of odd degrees, both the summand and the integrand in \eqref{eq:wigner_limits} are anti-symmetric, therefore $\cW^s[f]\equiv0$.

When $f(\uu)=u^{2d-l}v^l$ is a monomial of an even degree, the proof follows exactly as in \cite[Theorem 4]{wu2021corrected} for $\mathrm{Re}\,s < 2$. For $\mathrm{Re}\,s > 2$, note that
\begin{equation}
\int_{|\uu|<\infty}\frac{1}{Q(\uu)^{s/2}}\d\uu = |Q|^{-1/2}\int_{|\uu|<\infty}\frac{1}{(u^2+v^2)^{s/2}}\d\uu = |Q|^{-1/2}\,2\pi\int_{0}^\infty\frac{1}{r^{s-1}}dr=0
\end{equation}
in the finite-part sense, where $|Q| := EG-F^2$ is the determinant of $Q(\uu)$, therefore the Wigner limit \eqref{eq:wigner_limits} reduces to
\begin{equation}
\cW^s[1] = \sideset{}'\sum_{\ii}\frac{1}{Q(\ii)^\frac{s}{2}} = Z(s;Q), \quad \mathrm{Re}\,s > 2.
\end{equation}
Therefore $\cW^s[1]\equiv Z(s;Q)$ holds for all $s\neq2$ under analytic continuation. In particular, applying appropriate parametric derivatives to $\cW^{s-2n}[1] = Z(s-2n;Q)$ on both sides yields \eqref{eq:wigner_epstein}.
\end{proof}

The computation of the Wigner limits $W^s[\uu^{2n}]$ using the formulae \eqref{eq:wigner_epstein} requires evaluating the $n$-th derivatives of the Epstein zeta functions with respect to the parameters $E, F,$ and $G$. We include the procedure for computing the Epstein zeta derivatives in Appendix \ref{app:compute_epstein}.

Theorems \ref{thm:extendedEM2D} and \ref{thm:wigner} together provide a practical way to correct the Trapezoidal quadrature errors associated with integrands of the form \eqref{eq:generic_kernel}.

\paragraph{Example 1} The Laplace single-layer potential on a rectangular parametric domain has the form  :
\begin{equation}
\int_\Gamma\frac{\sigma(\yy)}{4\pi\,r}\,\d s_\yy \equiv \int_R\frac{\varphi(\uu)}{r(\uu)}\,\d\uu,\qquad \frac{\varphi(\uu)}{r(\uu)} \sim \frac{O(1)}{\sqrt{Q}}+\frac{O(\uu^3)}{\sqrt{Q}^3}+\frac{O(\uu^6)}{\sqrt{Q}^5}+\dots
\end{equation}
where the smooth function $\varphi(\uu):=J(\uu)\,\sigma(\uu)/(4\pi)$. An $O(h^P)$ quadrature using \eqref{eq:extendedEM2D} requires correcting the errors due to $\frac{\Theta(\uu^m)}{Q(\uu)^{n+1/2}}$ for all $p$ even and $3n\leq m < 2n+P-1$, $n = 0,1,\dots$. 

For an $O(h^3)$ quadrature ($P=3$), such $m$ exists only when $n = 0$, where $m=0$, so the corresponding term is $\frac{\varphi(\mathbf{0})}{\sqrt{Q}} = \frac{\Theta(1)}{\sqrt{Q}}.$
Then truncating \eqref{eq:extendedEM2D} gives
\begin{equation}
\int_R \frac{\varphi(\uu)}{r(\uu)}\,\d \uu = \sum_{\substack{0< j< M, 0< k< N\\ (u_j,v_k)\neq (0, 0)}} \frac{\varphi(u_j,v_k)}{r(u_j,v_k)}\,h^2  -  \varphi(\mathbf{0})Z(1;Q) h + O(h^3).
\end{equation}
On the other hand, an $O(h^5)$ quadrature requires correcting components with $m$ even and $3n\leq m < 2n+4$, $n = 0,1,2$, 
resulting in the quadrature
\begin{equation}
\begin{aligned}
&\int_R \frac{\varphi(\uu)}{r(\uu)}\,\d \uu = \sum_{\substack{0< j< M, 0< k< N\\ (u_j,v_k)\neq (0, 0)}} \frac{\varphi(u_j,v_k)}{r(u_j,v_k)}\,h^2  -  \varphi(\mathbf{0})Z(1;Q) h\\
& - h^3\left(\sum_{l=0}^{2} \alpha_{2,l}^1\cW^{1}\left[u^{2-l}v^l\right]+\sum_{l=0}^{4} \alpha_{4,l}^3\cW^{3}\left[u^{4-l}v^l\right]+\sum_{l=0}^{6} \alpha_{6,l}^5\cW^{5}\left[u^{6-l}v^l\right]\right) +  O(h^5).
\label{eq:slp_Oh5_raw}
\end{aligned}
\end{equation}
where the coefficients $\alpha_{p,l}^s$ in the $O(h^3)$ term involve $\varphi$ and its derivatives at $\bs\rho(\mathbf{0})$, which requires higher derivatives of the geometry; these expressions are a little involved but only need to be calculated once, they are given in \cite[section 5]{wu2021corrected}. 

The fact that higher derivatives of the geometry are required to compute the coefficients $\alpha_{p,l}^{s}$ makes it highly inconvenient to derive high-order quadrature schemes, which becomes even more tedious for different kernels other than the Laplace SLP; moreover, boundary data in practice are often given numerically, so the analytic expressions for the higher derivatives are unavailable anyway. Fortunately, all these difficulties can be avoided. Next, we develop a moment-fitting approach that allows easy construction of high-order quadrature rules for all kinds of elliptic kernels.

%sssssssssssss
\subsection{Moment fitting on local stencils}
Suppose one wants an $O(h^P)$ trapezoidal rule for the integrating \eqref{eq:generic_kernel}, it suffices to consider one term
\begin{equation}
\frac{\varphi\cdot(r^2-Q)^m}{Q^{(2m+p)/2}} = \frac{O(\uu^{2q+3m})}{Q^{(2m+p)/2}}.
\label{eq:one_term}
\end{equation}
By \eqref{eq:extendedEM2D}, a ${\Theta(\uu^{2k})}/{Q^{(2m+p)/2}}$ term in \eqref{eq:one_term} contributes errors of the form
\begin{equation}
\sum_{l=0}^{2k}\alpha_{k,l}\cW^{2m+p}\left[u^{2k-l}v^l\right]h^{2k-(2m+p)+2}
\label{eq:2k_moment}
\end{equation}
where the coefficients $\alpha_{k,l}$'s depend on $\varphi\cdot(r^2-Q)^m$ and its derivatives at $\uu=\mathbf{0}$; we call the errors \eqref{eq:2k_moment} the \emph{$2k$-moments}. So to achieve $O(h^P)$ accuracy one must correct all the $2k$-moments in \eqref{eq:one_term} where $2q+3m\leq 2k < P+p+2m-2$, or 
\begin{equation}
K_1\leq k\leq K_2,\quad \text{where } 
\begin{cases}
K_1 \equiv K_1(m,q) := q+\left\lceil\frac{3m}{2}\right\rceil\\
K_2 \equiv K_2(m,P,p) := \left\lceil\frac{P+p}{2}\right\rceil+m-2.
\end{cases}
\label{eq:bounds_for_k}
\end{equation}
In addition, $k$ exists in \eqref{eq:bounds_for_k} only when $K_1\leq K_2$, or $q+\lceil\frac{3m}{2}\rceil\leq \lceil\frac{P+p}{2}\rceil+m-2$, which implies that $m$ has an upper bound 
\begin{equation}
0\leq m\leq M, \quad\text{where } M:=2\left\lceil\frac{P+p}{2}\right\rceil-2q-4.
\label{eq:bounds_for_m}
\end{equation}
Therefore, one only has to correct finitely many $2k$-moments to get a $O(h^P)$ trapezoidal rule for integrating \eqref{eq:generic_kernel}. We next describe the moment-fitting procedure to fit all these error moments on a local stencil around the singular point.

\begin{figure}[htbp]
\centering
\includegraphics[width=0.8\textwidth]{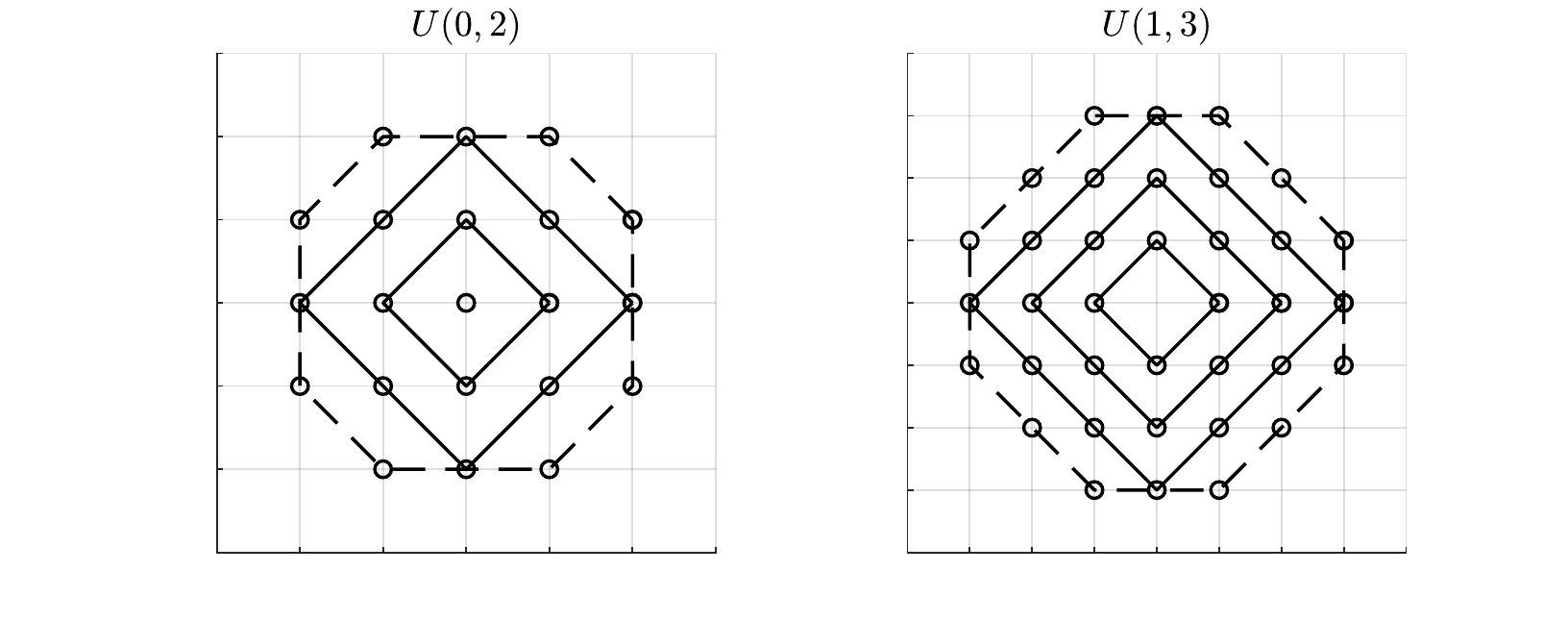}
\caption{Example correction stencils. Let ``layer $k$'' be the set of all the points in $\ZZ^2$ whose $1$-norm equals $k$, $k\geq0$. Then a stencil $U(K_1,K_2)$ includes layer $k$ for $K_1\leq k\leq K_2+1$ and excludes the four points on the axes in layer $(K_2+1)$. In the figure, the solid and dashed lines show the layers in $U(K_1,K_2)$.}\label{fig:new_stencil}
\end{figure}

\paragraph{Moment-fitting procedure} For a fixed $m$, the error moments \eqref{eq:2k_moment} associated with the term \eqref{eq:one_term} can be approximated to $O(h^P)$ accurate using the following steps.
\begin{enumerate}
\item For $K_1, K_2$ specified in \eqref{eq:bounds_for_k}, define the stencil $U(K_1,K_2)$ as
    \begin{equation}
    U(K_1,K_2):=\{(\mu,\nu)\in\ZZ\,:\,K_1\leq |\mu|+|\nu|\leq K_2+1, \max(|\mu|,|\nu|)\leq K_2\}
    \label{eq:stencil}
    \end{equation}
    See Figure \ref{fig:new_stencil} for example stencils.
\item Define moment-fitting weights $\tau_{\mu,\nu}$, $(\mu,\nu)\in U(K_1,K_2)$. We solve for $\tau_{\mu,\nu}$ by imposing two sets of conditions. Firstly, the moment-fitting conditions for $2k$-moments \eqref{eq:2k_moment} are
    \begin{equation}
    \sum_{(\mu,\nu)\in U(K_1,K_2)} (\mu h)^{2k-l_k}(\nu h)^{l_k}\tau_{\mu,\nu}\,h^{2-s} = -\cW^{s}\left[u^{2k-l_k}v^{l_k}\right]h^{2k-s+2}
    \label{eq:moment_fit_1}
    \end{equation}
    for $l_k = 0,1,\dots,2k$ and for $K_1\leq k\leq K_2$ as in \eqref{eq:bounds_for_k}, and where $s: = 2m+p$. Note that \eqref{eq:moment_fit_1} is independent of $h$ after cancelation. Secondly, we impose the following symmetry conditions
    \begin{equation}
    \begin{aligned}
    \tau_{\mu,\nu} &= \tau_{-\mu,-\nu} &&\text{for all }(\mu,\nu)\in U(K_1,K_2),\\
    \tau_{\mu,\nu} &= -\tau_{-\mu,\nu} &&\text{if }|\mu|+|\nu|=K_2+1,\\
    \tau_{\mu,\nu} &= \tau_{-\mu,\nu} &&\text{if }|\mu|+|\nu|=K_1.
    \end{aligned}
    \label{eq:moment_fit_2}
    \end{equation}
    It is not hard to see that both the number of points in $U(K_1,K_2)$ and the total number of equations in (\ref{eq:moment_fit_1}--\ref{eq:moment_fit_2}) are $2(K_1+K_2)(K_2-K_1+1)+4K_2+\delta_{0,K_1}$, thus the weights $\tau_{\mu,\nu}$ are well-defined.
\item Consequently, when integrating \eqref{eq:generic_kernel} using the trapezoidal rule, the error contribution from one term \eqref{eq:one_term} is given by the approximation
\begin{equation}
\sum_{(\mu,\nu)\in U(K_1,K_2)}\varphi(\mu h,\nu h)\,\Big(r(\mu h,\nu h)^2-Q(\mu h,\nu h)\Big)^m\,\tau_{\mu,\nu}\,h^{2-s}+ O(h^{P}).
\label{eq:zeta_corr_1}
\end{equation}
\end{enumerate}
Repeating the above procedure for $m$ in the range \eqref{eq:bounds_for_m}, the errors in the trapezoidal rule approximation of \eqref{eq:generic_kernel0} are corrected up to $O(h^P)$. We are ready to state the unified trapezoidal quadrature.

%ssssssssss
\subsection{Unified zeta quadrature for boundary integral operators}
\label{sc:unified_zeta_BIO3d}

%thmmmmmmmmmmmm
\begin{theorem}[Unified Zeta Quadrature]
\label{thm:unified_zeta3d}
Suppose $\varphi(\uu)$ is doubly-periodic on (or compactly supported in) $[-a,a]^2$ and is infinitely differentiable, $\varphi(\uu) = O(\uu^{2q})$ for some integer $q\geq0$, then
\begin{equation}
\int_{|\uu|\leq a}\frac{\varphi(\uu)}{r(\uu)^p}\,\d\uu = \sideset{}'\sum_{|\ii |\leq N}\frac{\varphi(\ii h)}{r(\ii h)^p}h^2 + C_p^h[\varphi;P,q] + O(h^P)
\label{eq:zetatrap_general3d}
\end{equation}
where $h=a/N$ and the correction formula $C_p^h[\varphi;P,q]$ is given by
\begin{equation}
    C_p^h[\varphi;P,q] = \sum_{m=0}^{M}\quad\sum_{(\mu,\nu)\in U(K_1^m,K_2^m)}\binom{-\frac{p}{2}}{m}\,\tilde{\varphi}_m(\mu h,\nu h)\,\tau^m_{\mu,\nu}\,h^{2-p-2m},
    \label{eq:zetatrap_general3d_corr}
\end{equation}
where $M=2\lceil\frac{P+p}{2}\rceil-2q-4$, where $\tilde{\varphi}_m(\uu):=[r(\uu)^2-Q(\uu)]^m\varphi(\uu)$, where the stencils $U(K_1^m,K_2^m)$ are defined in \eqref{eq:stencil} with $K_1^m = q+\lceil\frac{3m}{2}\rceil$ and $K_2^m= \lceil\frac{P+p}{2}\rceil+m-2$, and where the associated weights $\tau^m_{\mu,\nu}$ are the solution of \eqref{eq:moment_fit_1}--\eqref{eq:moment_fit_2}.
\end{theorem}

\begin{remark}
\label{rmk:partial_kern_indep}
Compared to \cite{wu2021corrected} which computes all the coefficients in the generalized Euler-Maclaurin formula \eqref{eq:extendedEM2D}, the quadrature \eqref{eq:zetatrap_general3d} used a new moment-fitting approach to avoid computing any higher derivatives of the smooth components $\varphi\cdot (r^2-Q)^m$; only the first derivatives of the geometry are needed to construct $Q(\uu)$ and the weights $\tau^m_{\mu,\nu}$. This fact significantly simplifies the derivation of high-order quadrature rules and enables flexible application to different integral operators. Note that in the quadrature \eqref{eq:zetatrap_general3d}, $\varphi$ is an arbitrary smooth function (which can also be vector-valued), thus all common boundary integral operators associated with elliptic PDEs can be handled by \eqref{eq:zetatrap_general3d} as is.
\end{remark}

\paragraph{Example 1 (continued)}  Consider again the $O(h^5)$ trapezoidal rule \eqref{eq:slp_Oh5_raw} for the Laplace SLP. Using \eqref{eq:zetatrap_general3d} with $p=1, q=0, P=5,$ and $\varphi(\uu)=J(\uu)\,\sigma(\uu)/(4\pi)$, we obtain a new $O(h^5)$ trapezoidal rule on the stencil $U(0,3)$ (37 points total).
\begin{equation}
\begin{aligned}
\int_{|\uu|\leq a}\frac{\varphi(\uu)}{r(\uu)}\d\uu &= \sideset{}'\sum_{|\ii |\leq N}\frac{\varphi(\ii h)}{r(\ii h)}h^2 + h\sum_{(\mu,\nu)\in U(0,1)} \varphi(\mu h,\nu h)\,\tau^0_{\mu,\nu}\\ 
& -\frac{1}{2h}\sum_{(\mu,\nu)\in U(2,2)}\varphi(\mu h,\nu h)\,\big(r(\mu h,\nu h)^2-Q(\mu h,\nu h)\big)\,\tau^1_{\mu,\nu}\\ 
& + \frac{3}{8h^3}\sum_{(\mu,\nu)\in U(3,3)}\varphi(\mu h,\nu h)\,\big(r(\mu h,\nu h)^2-Q(\mu h,\nu h)\big)^2\,\tau^2_{\mu,\nu}+O(h^5).
\end{aligned}
\label{eq:O(h^5)quad_slp_nogeodiv}
\end{equation}

In general, a $O(h^P)$ trapezoidal rule for the SLP is given by
\begin{equation}
\begin{aligned}
\int_{|\uu|\leq a}\frac{\varphi(\uu)}{r(\uu)}\d\uu &= \sideset{}'\sum_{|\ii |\leq N}\frac{\varphi(\ii h)}{r(\ii h)}h^2 + C_1^h[\varphi;P,0]+O(h^P)
\end{aligned}
\label{eq:O(h^2M+3)quad_slp_nogeodiv}
\end{equation}
where the local correction $C_1^h[\varphi;P,0]$ from \eqref{eq:zetatrap_general3d_corr} is given by
\begin{equation}
C_1^h[\varphi;P,0] = \sum_{m=0}^{2\lceil\frac{P-3}{2}\rceil}\quad\sum_{U(\lceil\frac{3m}{2}\rceil,\lceil\frac{P-3}{2}\rceil+m)}\binom{-\frac{1}{2}}{m}\,\tilde{\varphi}_m(\mu h,\nu h)\,\tau^m_{\mu,\nu}\,h^{1-2m}
\label{eq:correction_terms_lap3d_slp}
\end{equation}
where $\tau^m_{\mu,\nu}$ are constructed on $U(\lceil\frac{3m}{2}\rceil,\lceil\frac{P-3}{2}\rceil+m)$ by solving \eqref{eq:moment_fit_1}--\eqref{eq:moment_fit_2}. The total stencil $U(0,3\left\lceil\frac{P-3}{2}\right\rceil)$ are shown in the top row of Figure \ref{fig:stencil_lap}.

\paragraph{Example 2 (hypersingular Laplace)} Consider the Laplace hypersingular boundary integral operator
\begin{equation}
    \cH[\sigma](\xx)=\frac{1}{4\pi}\int_\Gamma \left(\frac{\nn_\mathbf{0}\cdot\nn(\uu)}{r^3}-\frac{3\mu_\mathbf{0}\,\mu(\uu)}{r^5}\right)\sigma(\yy)\,\d s_\yy
    \label{eq:lap3d_hyper}
\end{equation}
where, again, we have assumed $\yy=\rr(\uu)$ and $\xx=\rr(\mathbf{0})$, so $\nn(\uu)\equiv\nn_\yy, \nn_\mathbf{0}\equiv\nn_\xx$, and $\mu(\uu):=(\xx-\yy)\cdot\nn_\yy, \mu_\mathbf{0}:=(\xx-\yy)\cdot\nn_\xx$.
To derive the quadrature formula for \eqref{eq:lap3d_hyper}, rewrite it as
\begin{equation}
    \cH[\sigma](\xx) = \int_{|\uu|\leq a}\left(\frac{\varphi(\uu)}{r(\uu)^3}+\frac{\psi(\uu)}{r(\uu)^5}\right)\d\uu
    \label{eq:lap3d_hyper_rewrite}
\end{equation}
where $\varphi(\uu) = (\nn_\mathbf{0}\cdot\nn(\uu))J(\uu)\sigma(\uu)/(4\pi)$ and $\psi(\uu) = -3\mu_\mathbf{0}\mu(\uu)J(\uu)\sigma(\uu)/(4\pi)$ are smooth; note that $\varphi(\uu) = O(1)$ and $\psi(\uu)=O(\uu^4)$ (because $\mu_\mathbf{0}$ and $\mu(\uu)$ are both $O(\uu^2)$). Then applying \eqref{eq:zetatrap_general3d} to the two components in \eqref{eq:lap3d_hyper_rewrite} with $p=3, q=0$ and $p=5, q=2$, respectively, gives a $O(h^P)$ trapezoidal rule for the hypersingular \eqref{eq:lap3d_hyper}
\begin{equation}
\begin{aligned}
\int_{|\uu|\leq a}\left(\frac{\varphi(\uu)}{r(\uu)^3}+\frac{\psi(\uu)}{r(\uu)^5}\right)\d\uu &= \sideset{}'\sum_{|\ii |\leq N}\left(\frac{\varphi(\ii h)}{r(\ii h)^3} +\frac{\psi(\ii h)}{r(\ii h)^5}\right)h^2\\
&\quad + C_3^h[\varphi;P,0] + C_5^h[\psi;P,2] + O(h^P)
\end{aligned}
\label{eq:zetatrap_lap3d_hyper}
\end{equation}
where the correction terms are given by 
\begin{equation}
\begin{aligned}
C_3^h[\varphi;P,0] &= \sum_{m=0}^{2\lceil\frac{P-1}{2}\rceil}\quad\sum_{U(\lceil\frac{3m}{2}\rceil,\lceil\frac{P-1}{2}\rceil+m)}\binom{-\frac{3}{2}}{m}\tilde{\varphi}_m(\mu h,\nu h)\,\tau^m_{\mu,\nu}\,h^{-1-2m},\\
C_5^h[\psi;P,2] &= \sum_{n=0}^{2\lceil\frac{P-3}{2}\rceil}\quad\sum_{U(\lceil\frac{3m}{2}\rceil+2,\lceil\frac{P+1}{2}\rceil+m)}\binom{-\frac{5}{2}}{m}\tilde{\psi}_m(\mu h,\nu h)\,\tau^m_{\mu,\nu}\,h^{-3-2m},
\end{aligned}
\label{eq:correction_terms_lap3d_hyp}
\end{equation}
where $\tilde{\varphi}_m(\uu)=[r(\uu)^2-Q(\uu)]^m\varphi(\uu)$ and $\tilde{\psi}_m(\uu)=[r(\uu)^2-Q(\uu)]^m\psi(\uu)$, and where $\tau^m_{\mu,\nu}$ are constructed on the corresponding stencils by solving \eqref{eq:moment_fit_1}--\eqref{eq:moment_fit_2}.

\paragraph{Example 3 (Helmholtz SLP)} The Helmholtz SLP has the form
\begin{equation}
\cS_\kappa[\sigma](\xx) = \int_\Gamma \frac{e^{i\kappa r}}{4\pi r}\sigma(\yy)\,\d S_\yy \equiv  \int_R \frac{\varphi(\uu)}{r(\uu)}\,\d\uu \equiv  \int_R \frac{\varphi^r(\uu)+i\,\varphi^i(\uu)}{r(\uu)}\,\d\uu
\end{equation}
where the real and imaginary parts of $\varphi(\uu)  :=  e^{i\kappa\,r(\uu)}J(\uu)\sigma(\uu)/(4\pi)$ are
\begin{equation}
\varphi^r(\uu) := \cos(\kappa\,r(\uu))J(\uu)\sigma(\uu)/(4\pi),\qquad \varphi^i(\uu) := \sin(\kappa\,r(\uu))J(\uu)\sigma(\uu)/(4\pi).
\label{eq:helmslp_real_imag}
\end{equation}
Notice that the imaginary part $\varphi^i(\uu)/r(\uu)$ is smooth and with a diagonal limit
\begin{equation}
\varphi^i_0:=\lim_{\uu\to\mathbf{0}} \frac{i\,\varphi^i(\uu)}{r(\uu)} = \lim_{\uu\to\mathbf{0}}\frac{i\sin(\kappa\,r(\uu))}{4\pi\,r(\uu)}J(\uu)\sigma(\uu) = \frac{i\kappa}{4\pi}J(\mathbf{0})\sigma(\mathbf{0})
\label{eq:helmholtz_diag_lim}
\end{equation}
so one only needs to correct the quadrature error from the real part. Applying the ordinary trapezoidal rule to the imaginary part and the unified zeta quadrature to the real part, we obtain an $O(h^P)$ trapezoidal rule for the Helmholtz SLP
\begin{equation}
\begin{aligned}
\int_{|\uu|\leq a}\frac{\varphi(\uu)}{r(\uu)}\d\uu &= \sideset{}'\sum_{|\ii |\leq N}\frac{\varphi(\ii h)}{r(\ii h)}h^2 + \varphi^i_0h^2 +  C_1^h[\varphi^r;P,0]+O(h^P)
\end{aligned}
\label{eq:O(h^2M+3)quad_helmslp_nogeodiv}
\end{equation}
where $\varphi^i_0$ is given by \eqref{eq:helmholtz_diag_lim}, and where
\begin{equation}
C_1^h[\varphi^r;P,0] = \sum_{m=0}^{2\lceil\frac{P-3}{2}\rceil}\quad\sum_{U(\lceil\frac{3m}{2}\rceil,\lceil\frac{P-3}{2}\rceil+m)}\binom{-\frac{1}{2}}{m}\tilde{\varphi}^r_m(\mu h,\nu h)\,\tau^m_{\mu,\nu}\,h^{1-2m},
\label{eq:correction_terms_helm3d_slp}
\end{equation}
which is almost exactly the same as \eqref{eq:correction_terms_lap3d_slp} for the Laplace SLP, except that $\tilde{\varphi}^r_m(\uu) := [r(\uu)^2-Q(\uu)]^m\varphi^r(\uu)$ with $\varphi^r$ from \eqref{eq:helmslp_real_imag}.

\paragraph{Example 4} The derivations in Examples 1, 2 and 3 can be directly generalized to all common boundary integral operators associated with Laplace, Helmholtz, Stokes, and Maxwell equations. All of these operators consist of integrals of the form
$$\int\frac{\varphi(\uu)}{r(\uu)^p}\,\d\uu,$$
for which high-order trapezoidal rules can be constructed using the unified quadrature \eqref{eq:zetatrap_general3d}; in particular, the numerator $\varphi(\uu)$ can be a tensor function (e.g., Stokes and Maxwell), then one only needs to apply Theorem \ref{thm:unified_zeta3d} to each tensor component. We omit the repetitive derivations for these integral operators and will only present numerical results in the next section.

Figure \ref{fig:stencil_lap} shows the local stencils for the $O(h^3)$ to $O(h^9)$ trapezoidal rule corrections for the Laplace SLP and DLP. The stencils for the Laplace hypersingular operator \eqref{eq:lap3d_hyper} are identical to those for the SLP in the top row of Figure \ref{fig:stencil_lap}, except that the orders of accuracy achieved are $O(h^1)$ to $O(h^7)$ instead. We mention that these stencils are bigger than those used in \cite{wu2021corrected} because no information of the higher derivatives of the geometric parameterization is used.
    
%ffffffffffffffffff
\begin{figure}[htbp]
\centering
\includegraphics[width=\textwidth]{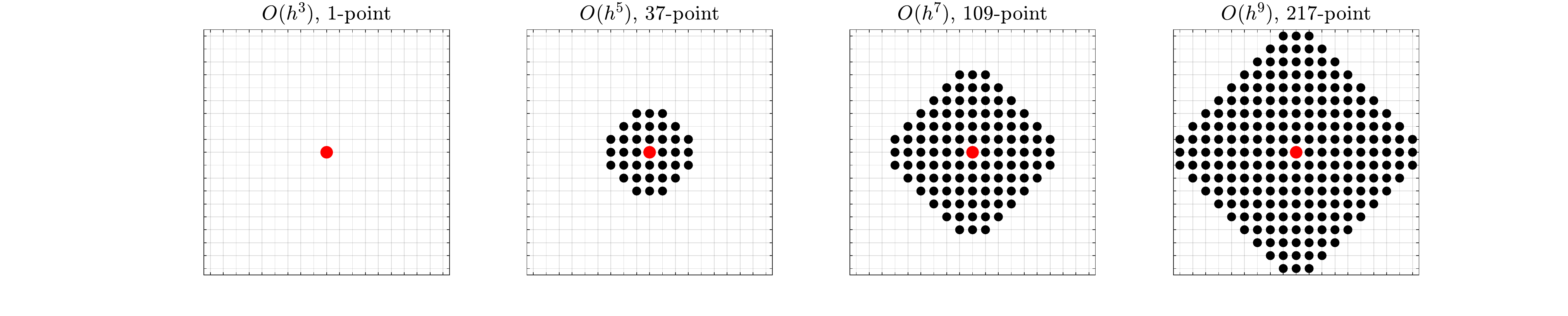}\\
\includegraphics[width=\textwidth]{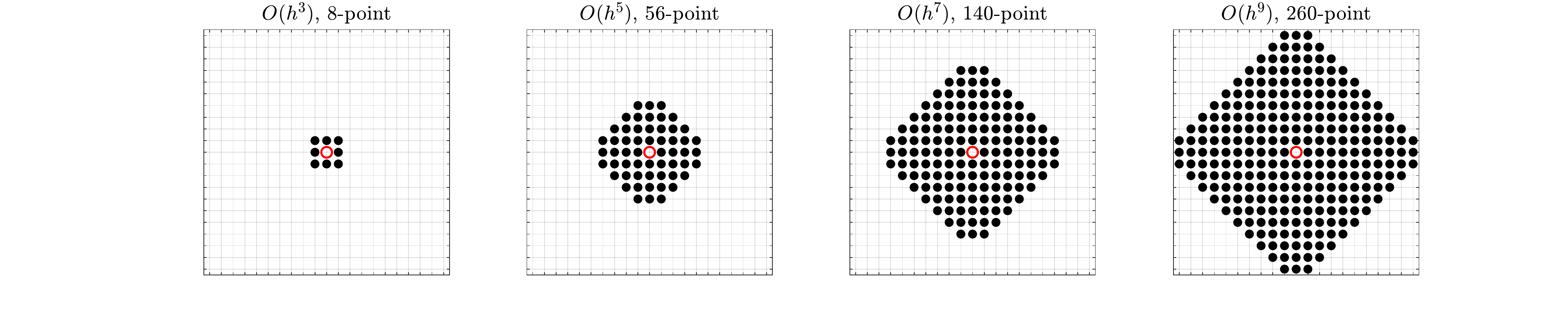}
\caption{$O(h^3)$ to $O(h^9)$ correction stencils for the Laplace SLP (top row) and DLP (bottom row), where a circle means that the correction weight is zero there. The normal derivative of the SLP uses the same stencils as the DLP (bottom row) and achieves the same orders of accuracy. The normal derivative of the DLP (i.e., the hypersingular \eqref{eq:lap3d_hyper}) uses the same stencils as the SLP (top row), but the achieved orders of accuracy are 2 orders less in each case (i.e., from $O(h^1)$ to $O(h^7)$) }\label{fig:stencil_lap}
\end{figure}

\subsection{Numerical results}
%sssssssss

We now show numerical results for the convergence of the unified zeta quadrature. All examples in this section are computed in MATLAB R2020a on a 10-core 3.7 GHz Intel Core i9 desktop.

Figure \ref{fig:patchZetaConv3d_laphelm} shows the convergence of the unified zeta quadrature applied to the Laplace and Helmholtz layer potentials and their normal derivatives on a randomly generated quartic surface patch centered at $\mathbf{0}$ and parameterized over $[-0.5,0.5]^2$, which is discretized on a uniform $h$-mesh. The density function is also randomly generated using the formula
\begin{equation}
    \sigma(u, v) = (a \cos(a + u) + b \sin(b + v)) e^{-c(u^2 +v^2)^4}
    \label{fig:patchDensSamp}
\end{equation}
where $a, b$ are standard Gaussian random numbers, and where $c = 2400$ such that $\sigma$ is compactly supported (up to double-precision) on the patch. Note that for both Laplace and Helmholtz, the normal derivative of the DLP (denoted DLPn) is hypersingular whereas the other three potentials are weakly singular, thus with the same amount of work the order of convergence for the hypersingular potential is 2 less than the other potentials. The bigger absolute errors for the hypersingular potential can be attributed to the fact that the hypersingular operator is essentially a differential operator, whose conditioning scales as $O(1/h)$.

%ffffffffffffffffff
\begin{figure}[htbp]
\includegraphics[width=\textwidth]{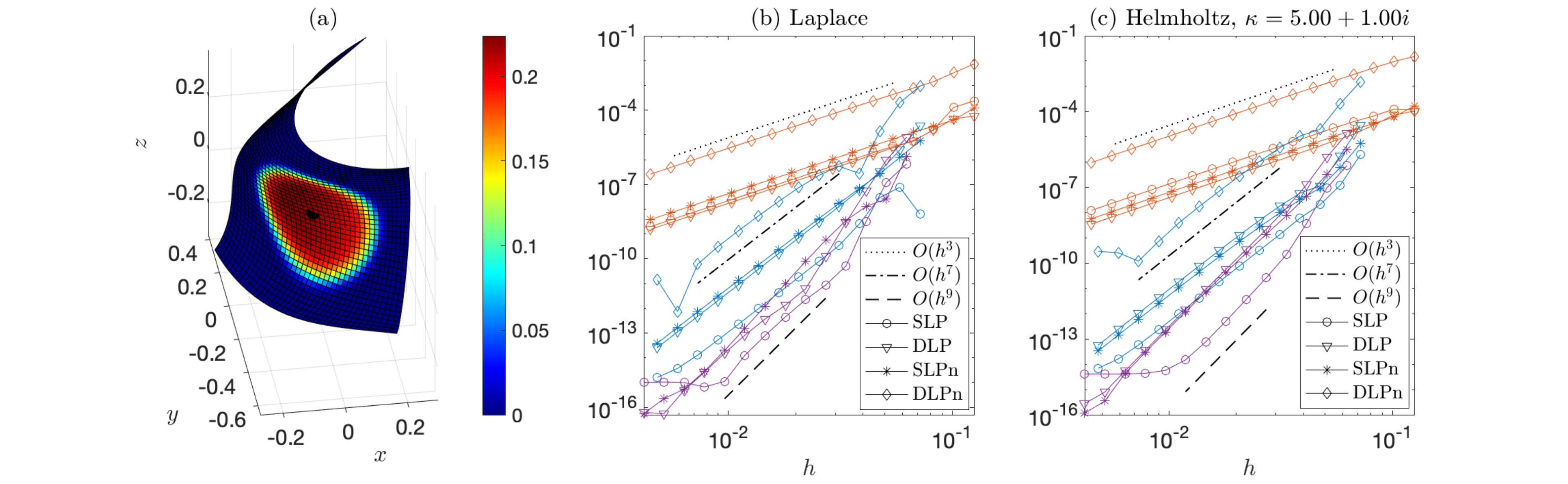}
\caption{(a) Evaluation of Laplace and Helmholtz potentials on a quartic surface patch using the unified zeta quadrature; target point located at the center (marked ``$\bullet$''), color represents the random density function \eqref{fig:patchDensSamp} with $a=0.22, b=-0.018$. Convergence against the grid size $h$ are shown in (b) for Laplace and (c) for Helmholtz. Convergence of $3^\text{rd}$ (red), $7^\text{th}$ (blue) and $9^\text{th}$ (purple) orders are shown for the SLP (circles), DLP (triangles), and the normal derivative of SLP (asterisks, denoted SLPn). Convergence up to $7^\text{th}$ order are shown for the normal derivative of DLP (diamonds, denoted DLPn), which is hypersingular. See Figure \ref{fig:stencil_lap} for the correction stencils used for these operators.}
\label{fig:patchZetaConv3d_laphelm}
\end{figure}

We next solve the Dirichlet and Neumann boundary value problems (BVPs) associated with the Laplace and Helmholtz equations exterior to a wobbly toroidal surface. The Laplace BVPs are reformulated as second kind integral equations based on potential theory \cite[\S6.4]{kress2014linear} and the Helmholtz BVPs as combined-field integral equations \cite[\S2]{bremer2015high}. All boundary integral equations are discretized using the Nystr\"om method with the unified zeta quadrature described in Section \ref{sc:unified_zeta_BIO3d}. The solution procedure is as follows: first the quadrature correction weights are pre-computed, then the integral equations are solved iteratively using GMRES with a tolerance $\epsilon_{\text{GMRES}}$; in each iteration, we first apply the punctured trapezoidal rule discretization of the Laplace/Helmholtz kernel using the FMM, then the pre-computed local correction weights are applied as a sparse matrix-vector multiplication. The overall computational cost will be $O(N)$ when $N$ discretization points are used.

The top row of Figure \ref{fig:bie3d_conv} shows the convergence of the relative errors and the timings for the solution of the Laplace BVPs. Convergence of $3^\text{rd}$ to $7^\text{th}$ orders are observed with the corresponding orders of quadrature corrections. Although higher-order correction weights require longer pre-computation times, the times for GMRES iterations are independent of the orders of the quadrature because the times to apply the sparse correction weights are negligible compared to applying the FMM; the overall solution times are clearly $O(N)$. Similarly, the bottom row of Figure \ref{fig:bie3d_conv} shows the results of solving the Helmholtz BVPs with the same geometry setup and using the standard combined-field integral equation formulation \cite{bremer2015high}, where the integral equation for the Neumann BVP is regularized based on the Caldr\'on projector theory \cite{hsiao2008boundary}. Quadrature corrections of up to $9^\text{th}$ order are applied and the corresponding orders of convergence are observed.

To test the quadrature on hypersingular operator, we solve the Helmholtz BVPs again without regularization. Table \ref{tbl:helm3d_bvps_stats} shows the numerical results of solving the Helmholtz BVPs on the same geometry as in Figure \ref{fig:bie3d_conv}. The integral equation for the Neumann problem is hypersingular whose condition number grows as $O(1/h)$ \cite[\S5.2]{sidi2014richardson}, so the number of GMRES iteration grows with $N$. It is possible to regularize the hypersingular BIE to improve its conditioning (see the numerical results in \cite{wu2021corrected}), but the condition number for the regularized BIE still grows with the wavenumber $\kappa$. On the other hand, the hypersingular BIE is more convenient for building Fast Direct Solvers which can be used as preconditioners, so it becomes more advantageous for problems with larger $\kappa$.
 
\begin{figure}[htbp]
    \centering
    \includegraphics[width=\textwidth]{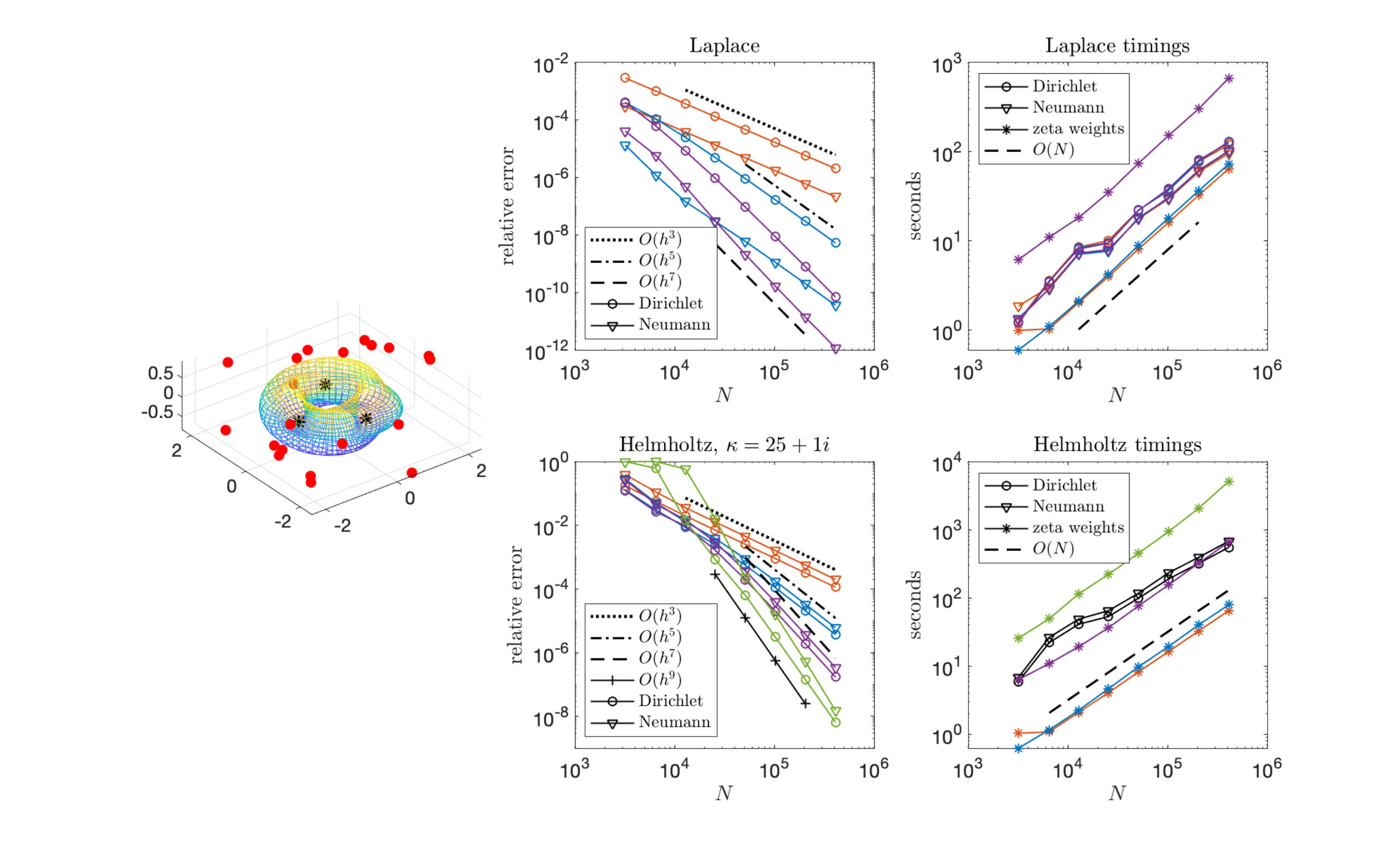}
    \caption{Convergence and timings for solving the Laplace (top row) and Helmholtz (bottom row) boundary value problems exterior to a wobbly torus using $N$ unknowns. The exact solutions are generated by three sources inside the torus (marked ``$\ast$'') and the max relative error is calculated at 20 target points outside the torus (marked ``$\bullet$''). Convergence of $3^\text{rd}$ (red), $5^\text{th}$ (blue), $7^\text{th}$ (purple)  and $9^\text{th}$ (green) orders are observed for both the Dirichlet (circles) and Neumann (triangles) problems, where the integral operators are discretized using the unified zeta quadrature of the corresponding order. The pre-computation times for the correction weights and for the iterative solves using GMRES and FMM are shown to scale linearly with $N$ in all cases.}
    \label{fig:bie3d_conv}
\end{figure} 

\begin{table}[htbp]
    \centering
    \def\arraystretch{1.2}
    \begin{tabular}{c|c|cccccc}
        \multicolumn{2}{c}{Helmholtz} & \multicolumn{6}{|c}{$3^\text{rd}$ order}\\
        \hline
        $N$ & $T_\text{wei}$ & $T_\text{iter}^\text{diri}$ & $N_\text{iter}^\text{diri}$ & $T_\text{iter}^\text{neu}$ & $N_\text{iter}^\text{neu}$ & $E^\text{diri}$ & $E^\text{neu}$ \\
        \hline
%        3200 	& 0.6 	& 5.19 	& 42 	& 10.27 	& 69 	& 5.5e-03 	& 2.3e-02 \\
%        6498 	& 1.1 	& 22.46 	& 38 	& 54.34 	& 94 	& 1.7e-03 	& 1.3e-02 \\
        12800 	& 2.1 	& 38.79 	& 37 	& 141.91 	& 131 	& 6.1e-04 	& 7.0e-03 \\
        25538 	& 4.2 	& 51.93 	& 37 	& 239.56 	& 168 	& 2.2e-04 	& 2.9e-03 \\
        51200 	& 8.8 	& 98.75 	& 37 	& 556.23 	& 205 	& 7.8e-05 	& 1.1e-03 \\
        102152 	& 17.4 	& 187.22 	& 37 	& 1307.74 	& 246 	& 2.8e-05 	& 4.1e-04 \\
        204800 	& 35.3 	& 322.44 	& 37 	& 2669.33 	& 293 	& 9.9e-06 	& 1.5e-04 \\
        410418 	& 69.9 	& 550.15 	& 37 	& 5495.41 	& 348 	& 3.5e-06 	& 5.3e-05 \\
        \hline
        \multicolumn{2}{c}{} & \multicolumn{6}{|c}{$5^\text{th}$ order}\\
        \hline
        $N$ & $T_\text{wei}$ & $T_\text{iter}^\text{diri}$ & $N_\text{iter}^\text{diri}$ & $T_\text{iter}^\text{neu}$ & $N_\text{iter}^\text{neu}$ & $E^\text{diri}$ & $E^\text{neu}$ \\
        \hline
%        3200 	& 6.3 	& 5.11 	& 42 	& 12.16 	& 79 	& 3.9e-03 	& 2.3e-02 \\
%        6498 	& 10.6 	& 21.60 	& 38 	& 62.56 	& 107 	& 9.3e-04 	& 1.2e-02 \\
        12800 	& 19.7 	& 40.03 	& 37 	& 162.99 	& 148 	& 2.6e-04 	& 4.3e-03 \\
        25538 	& 35.4 	& 53.87 	& 37 	& 272.63 	& 183 	& 7.8e-05 	& 1.1e-03 \\
        51200 	& 74.9 	& 100.45 	& 37 	& 621.05 	& 219 	& 1.7e-05 	& 2.2e-04 \\
        102152 	& 156.8 	& 195.31 	& 37 	& 1447.22 	& 259 	& 3.4e-06 	& 4.3e-05 \\
        204800 	& 318.8 	& 339.03 	& 37 	& 2972.25 	& 307 	& 6.3e-07 	& 8.0e-06 \\
        410418 	& 671.4 	& 582.13 	& 37 	& 6098.52 	& 363 	& 1.1e-07 	& 1.5e-06 \\
        \hline
        \multicolumn{2}{c}{} & \multicolumn{6}{|c}{$7^\text{th}$ order}\\
        \hline
        $N$ & $T_\text{wei}$ & $T_\text{iter}^\text{diri}$ & $N_\text{iter}^\text{diri}$ & $T_\text{iter}^\text{neu}$ & $N_\text{iter}^\text{neu}$ & $E^\text{diri}$ & $E^\text{neu}$ \\
        \hline
%        3200 	& 27.6 	& 5.58 	& 44 	& 105.15 	& 600 	& 3.8e-03 	& 6.2e-02 \\
%        6498 	& 50.5 	& 22.31 	& 38 	& 371.89 	& 600 	& 8.1e-04 	& 2.9e-02 \\
        12800 	& 112.2 	& 40.71 	& 37 	& 692.92 	& 600 	& 3.0e-04 	& 1.3e-02 \\
        25538 	& 212.4 	& 56.06 	& 37 	& 628.53 	& 395 	& 5.1e-05 	& 8.3e-04 \\
        51200 	& 453.2 	& 106.37 	& 37 	& 685.95 	& 228 	& 6.0e-06 	& 6.4e-05 \\
        102152 	& 921.5 	& 205.53 	& 37 	& 1579.91 	& 268 	& 6.2e-07 	& 6.1e-06 \\
        204800 	& 2138.3 	& 362.00 	& 37 	& 3279.47 	& 317 	& 5.8e-08 	& 6.2e-07 \\
        410418 	& 5227.4 	& 624.81 	& 37 	& 6801.70 	& 376 	& 5.3e-09 	& 6.0e-08
    \end{tabular}
    \caption{Convergence of solving the boundary value problems associated with the Helmholtz equation exterior to a wobbly torus (see Figure \ref{fig:bie3d_conv}) and with wavenumber $\kappa=25 + 1i$ ($12$ wavelengths across the geometry). The standard combined-field integral equation formulation is used with no regularization, thus the BIE for the Neumann problem is hypersingular. The superscripts ``diri'' and ``neu'' indicates Dirichlet and Neumann problems, respectively. $T_\text{wei}$ is the time for pre-computing the quadrature correction weights $\tau_{\mu,\nu}$ for all the involved integral operators: the DLP, SLP and their normal derivatives. $N_\text{iter}$ is the number of GMRES iterations and $T_\text{iter}$ is the total time for GMRES to converge to a residual of $10^{-12}$; each iteration consists of applying the punctured trapezoidal rule via FMM and applying the correction weights. $E$ in the last two columns denotes the relative $\infty$-norm error.}
    \label{tbl:helm3d_bvps_stats}
\end{table}

%ssssssssssssssssss
\section{Conclusion}
\label{sc:conclusion}

We have developed a new trapezoidal quadrature method for singular and hypersingular integral operators on curved surfaces in 2D and 3D. The quadrature method is a generalization of \cite{wu2021zeta} in 2D and of \cite{wu2021corrected} in 3D. In particular, the quadrature in 3D is based on a generalized Euler-Maclaurin formula (Theorem \ref{thm:extendedEM2D}) that provides an error expansion for a class of singular integrals in $\RR^2$; these errors are then fitted on a local stencil by a systematic moment-fitting approach. Compared to \cite{wu2021corrected}, this new quadrature significantly simplifies the derivation and construction of high-order accurate quadrature rules, and can be easily applied to all of the common integral operators for elliptic PDEs. In addition, the new quadrature also applies to hypersingular integral operators, which can be used to construct Fast Direct Solvers in a straightforward manner. On the other hand, a new algorithm is proposed (Appendix \ref{app:compute_epstein}) to compute the parametric derivatives of the Epstein zeta function to very high order, which is a core component for computing the high-order error expansions for our quadrature.

We have implemented up to $9^\text{th}$ order quadrature corrections for weakly singular operators and up to $7^\text{th}$ order for hypersingular operators in 3D. Our quadrature method is shown to be highly compatible with fast algorithms such as the FMM, achieving an $O(N)$ overall computational complexity when solving Laplace or Helmholtz BVPs on a curved surface.

Codes that accompany this paper are available on GitHub at the following repositories. 
\begin{center}
    \url{https://github.com/bobbielf2/ZetaTrap2D}\\
    \url{https://github.com/bobbielf2/ZetaTrap3D_Unified}
\end{center}

We propose two directions for future investigation. Firstly, the current quadrature method is restricted to surfaces that can be smoothly parameterized on a rectangle, such as a deformed torus. But one can potentially handle more general surfaces by combining our method with appropriate domain decomposition techniques, such as the partition-of-unity approach of \cite{bruno2001fast}. Secondly, when the target point is off but close to the surface, the BIOs become \emph{nearly singular}. It is possible to generalize our method to develop near-singular quadrature using a similar error correction approach, such as the recent development in \cite{nitsche2021evaluation} for line integrals. We will report on our investigation in these directions in the future.

\appendix

\section{Zeta quadratures for the Stokes potentials}
\label{app:sto2d}
The Stokes single- and double-layer velocity, pressure, and traction kernels have different types of singularity (see e.g. \cite{wu2020solution} for the definitions), which can be corrected using the zeta quadratures similar to the Laplace and Helmholtz layer potentials. Figure \ref{fig:sto2d_conv} shows the convergence of $16^\text{th}$-order zeta quadratures applied to all six 2D Stokes layer potentials evaluated at a target point located on the source curve. Details of derivations of formulae are omitted; codes for generating the figure is available on GitHub (link included in Section \ref{sc:conclusion}).

\begin{figure}[htbp]
\centering
\includegraphics[width=\textwidth]{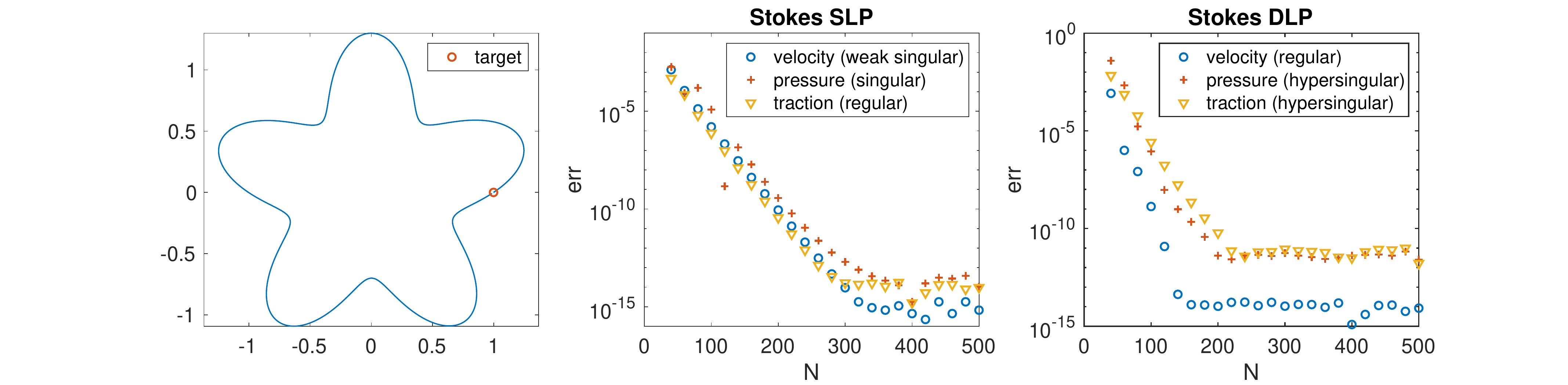}
\caption{Convergence of $16^\text{th}$-order zeta quadratures applied to the self-evaluations of the Stokes single- and double-layer velocity, pressure and traction on a smooth curve.}
\label{fig:sto2d_conv}
\end{figure}

\section{Computing the derivatives of the Epstein zeta functions}
\label{app:compute_epstein}

The computation of the Wigner limits requires higher derivatives of the Epstein zeta functions $Z(s;E,F,G)$ with respect to the parameters $E,F,$ and $G$. In this section we present the computational procedure for evaluating the Epstein zeta functions and their parametric derivatives.

We use $\sum_{i,j}'$ to denote the summation over the lattice $(i,j)\in\ZZ^2 - \{\mathbf{0}\}$. For any positive definite quadratic form $Q(u,v)\equiv Q(u,v; E,F,G) = Eu^2+2Fuv+Gv^2$, assume that the determinant $D := EG-F^2 = 1$. Then the analytic continuation of Epstein zeta function $Z(s;Q) \equiv Z(s;E,F,G) = \sum_{i,j}' Q(i,j)^{-\frac{s}{2}}$ from $\mathrm{Re}\,s>2$ to the whole complex plane (except a simple pole at $s = 2$) is given by the following integral representation \cite[Eq.(1.2.8),(1.2.11)]{borwein2013lattice}:
\begin{equation}
\begin{aligned}
\pi^{-s_1}\Gamma(s_1)Z(s) = -\frac{1}{s_1} - \frac{1}{s_2} + \int_1^\infty\mathrm{d}t\cdot t^{s_1-1}\sum_{i,j}{}'e^{-\pi Q(i,j) t}\\
+\int_1^\infty\mathrm{d}t\cdot t^{s_2-1}\sum_{i,j}{}'e^{-\pi \overline{Q}(i,j) t},
\end{aligned}
\label{eq:epstein_det1}
\end{equation}
where $s_1:=s/2$, $s_2:=1-s_1$, and $\overline{Q}(u,v):=Gu^2-2Fuv+Ev^2$. In the cases where the determinant $D\neq1$, one can first rescale $E,F,G$ by a factor of $1/\sqrt{D}$ when evaluating $Z(s)$ using the above formula, and then scale the final result by $D^{-s/4}$, therefore for a general quadratic form $Q(u,v)$ with a determinant $D$, the evaluation of $Z(s;Q)$ is given by
\begin{equation}
Z(s;E,F,G) = \frac{1}{D^{s/4}}Z(s; \tfrac{E}{\sqrt{D}},\tfrac{F}{\sqrt{D}},\tfrac{G}{\sqrt{D}})
\label{eq:epstein_scale}
\end{equation}
where the $Z(s)$ on the right-hand side is evaluated using the formula \eqref{eq:epstein_det1}.

Due to rotational symmetry, the sum involving $\overline{Q}(i,i)$ in the second integral of \eqref{eq:epstein_det1} satisfies
$$\sum_{i,j}{}'e^{-\pi \overline{Q}(i,j) t} = \sum_{i,j}{}'e^{-\pi Q(i,j) t}.$$
Therefore, combining with the formulae \eqref{eq:epstein_det1} and \eqref{eq:epstein_scale}, an expression for $Z(s)=Z(s;E,F,G)$, with a general determinant $D=EG-F^2\neq1$, is given by
\begin{equation}
\boxed{Z(s;Q) = C(s_1)\left(-\frac{1}{s_1s_2} + \sump_{i,j}\cG(\tilde{Q};s)\right)}
\label{eq:epstein_anal}
\end{equation}
where $s_1=s/2$, $s_2=1-s_1$, where the scaling factor $C(s_1)$ and the scaled quadratic form $\tilde{Q}$ are defined as
\begin{equation}
C(s_1)\equiv C(s_1;D):=\frac{\Gamma(s_1)}{(\pi\sqrt{D})^{s_1}},\qquad \tilde{Q}\equiv \tilde{Q}(i,j):= \frac{\pi Q(i,j)}{\sqrt{D}},
\label{eq:Cs_tQ}
\end{equation}
and where $\cG(x;s)$ a ``combined'' incomplete gamma function defined as
\begin{equation}
\cG(x)\equiv \cG(x;s):= \Gamma(s_1,x)x^{-s_1} + \Gamma(s_2,x)x^{-s_2} \equiv \int_1^\infty t^{s_1-1}e^{-xt}\,\d t + \int_1^\infty t^{s_2-1}e^{-xt}\,\d t
\label{eq:my_mod_gamma}
\end{equation}

Because $\cG(x;s)$ decays exponentially with $x$, the expression \eqref{eq:epstein_anal} gives a fast converging formula for evaluating $Z(s)$ when the sum therein is truncated appropriately based on the given tolerance; details can be found in \cite{wu2021corrected}.

\subsection{Mixed derivatives of the Epstein zeta functions}

Define the $k$-th derivative operator associated with coefficients $L, M, N$ as
\begin{equation}
\square^k \equiv \square^k_{(L,M,N)} := \left(L\frac{\partial}{\partial E} + M\frac{\partial}{\partial F}+N\frac{\partial}{\partial G}\right)^k \equiv (L\partial_E + M\partial_F+N\partial_G)^k
\end{equation}
we are interested in computing  $Z^{(k)}(s)\equiv \square^k Z(s)$ by differentiating \eqref{eq:epstein_anal}. We list the major steps for the derivation of the formulae for $Z^{(k)}(s)$ as follows.
\begin{itemize}
\item Because computing the derivatives of $Z(s)$ requires higher-order chain rules (e.g. the Fa\`a di Bruno's formula), we introduce the partial Bell polynomials
\begin{equation}
B_{n,m}(x_1,...,x_{n-m+1}) := \sum\frac{n!}{j_1!j_2!\cdots j_{n-m+1}!}
\left(\frac{x_1}{1!}\right)^{j_1}\left(\frac{x_2}{2!}\right)^{j_2}\cdots\left(\frac{x_{n-m+1}}{(n-m+1)!}\right)^{j_{n-m+1}}
\label{eq:bell_poly}
\end{equation}
where the sum is over all non-negative $j_1,\dots,j_{n-m+1}$ such that
\begin{align*}
    &j_1 + j_2 + \cdots + j_{n-m+1} = m,\\
    &j_1 + 2 j_2 + 3 j_3 + \cdots + (n-m+1)j_{n-m+1} = n.
\end{align*}
The Bell polynomials allow easier representations of higher-derivatives of composite functions. The following recurrence relation is used to compute $B_{n,m}$ given $x_1,\dots,x_{n-m+1}$
\begin{equation}
B_{n,m} = \sum_{i=1}^{n-m+1} \binom{n-1}{i-1} x_i B_{n-i,m-1}
\end{equation}
See \cite{johnson2002curious} for more details of the Fa\`a di Bruno's formula and Bell polynomials.
\item First consider the scaling factor $C(s_1)$ in \eqref{eq:epstein_anal}. Note that
\begin{equation}
\square C(s_1)\equiv  \square_{(L,M,N)} C(s_1;D) = -s_1H C(s_1)
\end{equation}
where the coefficient $H$ is defined as
\begin{equation}
H\equiv H(D;L,M,N):=\frac{GL+EN-2FM}{2D}.
\end{equation}
The derivatives $H^{(k)}\equiv\square^kH$ can be computed using the following recurrence relation
\begin{equation}
\begin{aligned}
    H^{(k)} &= -2(k-1)HH^{(k-1)}-(k-1)(k-2)KH^{(k-2)},\qquad k\geq2\\
    H^{(0)} &= H,\qquad H^{(1)}= -2H^2+K,
\end{aligned}
\label{eq:Hk}
\end{equation}
where $K\equiv K(D;L,M,N) :=  (LN-M^2)/D$.
\item If we denote $\tH = s_1H$, since $(\square+\tH)C(s_1) \equiv 0$, we have
\begin{equation}
(\square+\tH)\Big(C(s_1)\,f(E,F,G)\Big) = C(s_1)\,\square f(E,F,G)
\end{equation}
holds for any function $f$ that depends on $E,F,$ and $G$. In particular, when applied to the formula \eqref{eq:epstein_anal}, we have
\begin{equation}
(\square+\tH)^{k}Z(s) = C(s_1)\sump \cG^{(k)} \qquad k\geq1
\label{eq:dZ_simplified}
\end{equation}
where $\cG^{(k)}\equiv\square^k\cG(\tilde{Q})$.
\item To find the expressions for the derivatives $\cG^{(k)}$, define $\cG_k$, a shifted version of the combined incomplete gamma function \eqref{eq:my_mod_gamma}, as
\begin{equation}
\cG_k(x):= \Gamma(s_1+k,x)x^{-(s_1+k)} + \Gamma(s_2+k,x)x^{-(s_2+k)},
\label{eq:my_mod_gamma_k}
\end{equation}
so $\cG_0(x) \equiv \cG(x)$. Then using identities for the incomplete gamma function, we have $\cG_k'(x) = -\cG_{k+1}(x)$, thus by induction
\begin{equation}
\frac{d^i}{dx^i} \cG(x) = (-1)^i \cG_i(x),\qquad i = 0,1,2,\dots
\end{equation}
Then by the chain rule
\begin{equation}
\cG^{(k)} \equiv \square^k\cG(\tilde{Q}) = \sum_{i=1}^k (-1)^i\cG_i\,B_{k,i}(\tilde{Q}^{(1)},\tilde{Q}^{(2)},\dots,\tilde{Q}^{(k-i+1)})
\label{eq:Gk}
\end{equation}
where $\tQ^{(k)}\equiv \square^k\tQ$, and where $B_{k,i}$ are the partial Bell polynomials \eqref{eq:bell_poly}.
\item To find $\tilde{Q}^{(k)}\equiv\square^k_{(L,M,N)}\tilde{Q}$, define an auxiliary quadratic form $R(i,j)$ and its scaled form $\tilde{R}$ as follows
\begin{equation}
R(i,j)\equiv R(i,j; L,M,N) := Li^2+2Mij+Nj^2,\qquad \tilde{R}\equiv \tilde{R}(i,j):=\frac{\pi R(i,j)}{\sqrt{D}},
\end{equation}
then
\begin{equation}
\tilde{Q}^{(k)} = [(\square-H)^k\cdot1]\,\tilde{Q} + k\,[(\square-H)^{k-1}\cdot1]\,\tilde{R},
\label{eq:Qk}
\end{equation}
where the $1$'s in the square brackets represent constant functions.
\item We can expand the $(\square+\tH)^{k}Z(s)$ in \eqref{eq:dZ_simplified} using the chain rule, which gives
\begin{equation}
(\square+\tH)^{k}Z(s) = Z^{(k)}(s) + \sum_{i=0}^{k-1}\binom{n}{i}[(\square + \tH)^{k-i}\cdot1]Z^{(i)}(s),\qquad k = 1,2,\dots
\label{eq:Zk_expan}
\end{equation}
\item Finally, (\ref{eq:Qk},\ref{eq:Zk_expan}) require computing $(\square - H)^k\cdot1$ and $(\square + \tH)^k\cdot1\equiv(\square + s_1H)^k\cdot1$, which using the chain rule are given by the following formulae
\begin{equation}
\begin{aligned}
(\square - H)^k\cdot1 &= \sum_{i=1}^k(-1)^iB_{k,i}(H, H^{(1)},H^{(2)},\dots,H^{(k-i)})\\
(\square + s_1H)^k\cdot1 &= \sum_{i=1}^ks_1^iB_{k,i}(H, H^{(1)},H^{(2)},\dots,H^{(k-i)})
\end{aligned}
\label{eq:sqH}
\end{equation}
where $B_{k,i}$ are the partial Bell polynomials \eqref{eq:bell_poly}.
\end{itemize}

We are now ready to state the formulae for evaluating $Z^{(k)}(s)$. Combining \eqref{eq:dZ_simplified} and \eqref{eq:Zk_expan} we have the recursion relation
\begin{equation}
\boxed{
\begin{aligned}
Z^{(0)}(s) &\equiv Z(s) = C(s_1)\left(-\frac{1}{s_1s_2} + \sump_{i,j}\cG(\tilde{Q};s)\right) & s_1=\frac{s}{2}, s_2=1-s_1\\
Z^{(k)}(s) &= C(s_1)\sump_{i,j} \cG^{(k)} - \sum_{i=0}^{k-1}\binom{n}{i}[(\square + s_1H)^{k-i}\cdot1]Z^{(i)}(s), &k\geq1
\end{aligned}
}
\end{equation}
where $\cG$, $\cG^{(k)}$, and other required quantities, such as $(\square + s_1H)^{k}\cdot1$, can be evaluated using the formulae (\ref{eq:Cs_tQ}, \ref{eq:my_mod_gamma}, \ref{eq:my_mod_gamma_k}, \ref{eq:Gk}, \ref{eq:Qk}, \ref{eq:sqH}).

%sssssss
\subsection{Computing all parametric partial derivatives of the Epstein zeta}
Let $\pd_E := \frac{\pd}{\pd E}, \pd_F := \frac{\pd}{2\pd F}, \pd_G := \frac{\pd}{\pd G}$, we want the $n$-th derivatives $D^nZ$, $n=1,\dots,7$, where
\begin{align*}
D^1 &= \begin{Bmatrix}\pd_E, & \pd_F, & \pd_G\end{Bmatrix}\\
D^2 &= \begin{Bmatrix}\pd_E^2, & \pd_E\pd_F, & \pd_F^2, & \pd_F\pd_G, & \pd_G^2\end{Bmatrix}\\
D^3 &= \begin{Bmatrix}\pd_E^3, & \pd_E^2\pd_F, & \pd_E\pd_F^2, & \pd_F^3, & \pd_F^2\pd_G, & \pd_F\pd_G^2, & \pd_G^3\end{Bmatrix}\\
&\vdots\\
D^7 &= \begin{Bmatrix}\pd_E^7, & \pd_E^6\pd_F, & \dots, & \pd_F^7, & \dots, & \pd_F\pd_G^6, & \pd_G^7\end{Bmatrix}
\end{align*}
But the scheme in the previous section only computes the mixed derivatives $(a\pd_E+b\pd_F+c\pd_G)^nZ$ for any constants $a,b,c$. We will need a procedure to find all the partial derivatives using appropriate mixed derivatives. Since $\pd_E^{n-k}\pd_F^k$ and $\pd_G^{n-k}\pd_F^k$ are symmetric, a procedure for finding $\pd_E^{n-k}\pd_F^k, k=0,1,\dots,n$, will also work for finding $\pd_G^{n-k}\pd_F^k$. Here is the procedure for computing $\pd_E^{n-k}\pd_F^k$.
\begin{enumerate}
\item Let $a_m = \cos(\tfrac{m\pi}{2n})$ and $b_m = \sin(\tfrac{m\pi}{2n})$
\item Evaluate $c_m = (a_m\pd_E+b_m\pd_F)^nZ,\; m = 0,1,\dots,n$
\item We can now solve the system of $n+1$ equations
\begin{equation}
\sum_{k=0}^n \binom{n}{k}a_m^{n-k}b_m^k(\pd_E^{n-k}\pd_F^k)Z = c_m,\quad m=0,\dots,n
\label{eq:mixed_to_partial_Z}
\end{equation}
for the $n+1$ unknowns $(\pd_E^{n-k}\pd_F^k)Z$.
\item Since $a_m$ and $b_m$ only depend $n$, the inverse matrix for the system \eqref{eq:mixed_to_partial_Z} can be precomputed for any $n$.
\end{enumerate}

\section*{Acknowledgments}
The authors would like to thank Alex Barnett, Shravan Veerapaneni and Min Hyung Cho for several useful conversations.

%----------------------------------------------------------------------------------------------------------------------------------------------------------------
\bibliography{bobi_quadr}
\bibliographystyle{plain}

\end{document}